\documentclass[11pt]{article}
\usepackage{amsmath,amssymb,amsthm}
\usepackage{graphics,epsfig}
\usepackage{hyperref}
\usepackage[numbers]{natbib}
\usepackage{color}
\usepackage{graphicx}
\usepackage{caption}
\usepackage{subcaption}
\usepackage{float}
\usepackage{mathrsfs}
\usepackage{booktabs}
\usepackage{multirow}
\usepackage{comment}
\usepackage{setspace}
\usepackage{bbm}
\usepackage{bm}
\usepackage{amsfonts}
\usepackage{xcolor}
\usepackage{pslatex}
\usepackage{setspace}
\usepackage{bbm}
\usepackage{listings}
\def \aa {\boldsymbol{a}}

\def \hh {\boldsymbol{h}}

\def \ll {\boldsymbol{l}}

\def \pp {\boldsymbol{p}}

\def \uu {\boldsymbol{u}}
\def \vv {\boldsymbol{v}}
\def \ww {\boldsymbol{w}}
\def \xx {\boldsymbol{x}}
\def \yy {\boldsymbol{y}}
\def \zz {\boldsymbol{z}}

\def \UU {\boldsymbol{U}}

\def \XX {\boldsymbol{X}}
\def \YY {\boldsymbol{Y}}
\def \ZZ {\boldsymbol{Z}}

\def \Re {\mathbb{R}}

\newtheorem{definition}{\bf Definition}[section]
\newtheorem{proposition}{\bf Proposition}[section]
\newtheorem{theorem}{\bf Theorem}[section]

\newtheorem{lemma}{\bf Lemma}[section]
\newtheorem{corollary}{\bf Corollary}[section]

\newtheorem{remark}{\bf Remark}[section]

\newtheorem{example}{\bf Example}[section]


\begin{document}

\title{\bfseries  Lattice-Based Characterizations and Inference for Multivariate Dependence}

 \author{
    Paramahansa Pramanik$^{1,2}$
}

\date{
    \small
    $^{1}$Department of Mathematics and Statistics, University of South Alabama, Mobile, AL 36688, United States.\\
    $^{2}$E-mail address: \texttt{{\color{blue}ppramanik@southalabama.edu}}
}

%\date{\today}
\maketitle

\begin{abstract}
	
	\noindent
We develop a measure-theoretic framework for conditional
multivariate total positivity of order two
($\mathrm{cmMTP}_2$), formulated through lattice
inequalities without requiring the existence of
probability densities. This formulation extends
classical density-based notions of positive dependence
to general probability measures, including singular
distributions. We establish fundamental preservation
properties under marginalization, monotone
transformations, and weak convergence, thereby
providing a unified treatment of dependence structures
under distributional limits. We further investigate
the behavior of conditional total positivity under
Markov-kernel composition and derive characterizations
for important copula families. In the Archimedean
setting, the defining condition is expressed through
log-convexity of generator derivatives, whereas for
Gaussian distributions it corresponds to nonpositive
off-diagonal entries of the precision matrix, or
equivalently, nonnegative partial correlations.
To connect these structural results with statistical
inference, we introduce empirical lattice contrasts
and establish their joint asymptotic normality.
The resulting simultaneous testing procedures achieve
asymptotic size control and admit a local power
analysis. Our results provide a general foundation
for density-free characterization, specification
testing, and inference concerning positive dependence
in multivariate probability models, particularly
when conventional density-based methods are
inapplicable or insufficient.
	
\end{abstract}



\section{Introduction}
\label{sec:introduction}

Stochastic dependence among multiple random variables can be
characterized either through quantitative measures of association
or through classes of probability laws satisfying prescribed
dependence inequalities, with the former describing the strength
and direction of particular relationships and the latter identifying
structural properties of joint distributions, including positive
lower orthant dependence, stochastic monotonicity, and multivariate
total positivity of order two. Such structural dependence properties
play a fundamental role in multivariate statistics because they
impose mathematically interpretable restrictions on joint probability
distributions and their associated conditional dependence structures,
thereby providing a systematic framework for investigating positive
association, distributional ordering, and the behavior of stochastic
systems under conditioning and transformations. Among these notions,
multivariate total positivity of order two ($\mathrm{MTP}_2$)
occupies a particularly important position because of its underlying
lattice structure, its strong implications for positive dependence,
and its preservation properties under various probabilistic
operations, which make it useful for both theoretical investigations
and statistical applications
\citep{KarlinRinott1980a,MullerStoyan2002,FallatEtAl2017}.
More recently, \citet{FuchsWang2027} introduced multivariate
total positivity of order two for conditional distributions,
denoted $\mathrm{cMTP}_2$, extending the classical framework
by explicitly incorporating conditioning variables into the
formulation of total positivity and establishing a dependence
concept that lies between density-based and distribution-function
total positivity. This conditional perspective provides a
more refined description of multivariate dependence by accounting
for the influence of conditioning information on the lattice
inequalities governing the underlying probability law.

Motivated by these developments, we introduce and systematically
investigate a measure-theoretic counterpart of conditional
multivariate total positivity of order two, denoted
$\mathrm{cmMTP}_2$, which preserves the essential conditional
lattice structure while eliminating the requirement that the
underlying probability distribution admit a density with respect
to a product measure. In contrast to distribution-function
$\mathrm{MTP}_2$, which imposes total positivity restrictions
on joint lower-orthant probabilities, the proposed
$\mathrm{cmMTP}_2$ framework incorporates additional lattice
inequalities involving arbitrary admissible conditioning events,
thereby capturing conditional dependence restrictions that are
not expressed solely through the joint distribution function.
The measure-theoretic formulation is particularly relevant
because these conditional lattice inequalities remain
well-defined and mathematically meaningful for continuous,
discrete, mixed, and singular probability laws, including
settings in which conventional density-based characterizations
are unavailable or inappropriate. Consequently, the proposed
framework provides a unified mathematical foundation for
studying conditional dependence across different classes
of multivariate probability distributions without imposing
unnecessary absolute continuity assumptions. Beyond the
formulation of the dependence property itself, this approach
establishes a common setting for investigating the behavior
of conditional total positivity under weak convergence,
examining its compatibility with Markov composition,
developing copula representations and characterizations,
and constructing density-free procedures for multivariate
statistical inference. Taken together, these directions
connect the structural theory of conditional positive
dependence with probabilistic stability, stochastic
transition mechanisms, copula-based modeling, and
inferential methods applicable to general probability laws.

The theory of total positivity originates in the classical work
of Karlin and was developed probabilistically by
\citet{KarlinRinott1980a}; see also \citet{MullerStoyan2002}
for stochastic dependence and \citet{FallatEtAl2017} for
$\mathrm{MTP}_2$ in Markov structures. Density-based total
positivity has subsequently been investigated in graphical
models, exponential families, and statistical estimation
\citep{LauritzenEtAl2021,RobevaEtAl2021}. A complementary
approach formulates positive dependence directly through
probability measures. In particular, \citet{ColangeloEtAl2006}
introduced a density-free formulation of $\mathrm{MTP}_2$
through lattice inequalities for probabilities of Borel sets
and established its stability under weak convergence.
Conditional total positivity has a more recent development.
\citet{FuchsTschimpke2023} investigated conditional
total positivity from a Markov-kernel perspective in the
bivariate setting, while \citet{FuchsWang2027} introduced
$\mathrm{cMTP}_2(\XX_B\mid\XX_A)$ for multivariate
distributions and studied its relations with
$\mathrm{dMTP}_2$, $\mathrm{MTP}_2$, stochastic
monotonicity, Archimedean copulas, and Markov products.
Their formulation, however, requires
$\XX_A$ to admit a density with respect to
$\bigotimes_{i\in A}P_{X_i}$, thereby excluding
certain singular dependence structures and limiting
its applicability under weak convergence when absolute
continuity is lost. A density-free conditional analog
and its closure properties were consequently identified
as directions for further research \citep{FuchsWang2027}.
The present paper addresses these questions by combining
the measure-theoretic approach of
\citet{ColangeloEtAl2006} with the conditional framework
of \citet{FuchsWang2027}. Related developments include
weak-convergence methods \citep{Billingsley1999},
Archimedean positive-dependence characterizations
\citep{MullerScarsini2005,McNeilNeslehova2009},
and copula representations \citep{DuranteSempi2015}.

Our first contribution is the introduction and characterization
of \emph{conditional measure-theoretic multivariate total
	positivity of order two}. For a non-trivial partition
$A\dot\cup B=[d]$, the property
$\mathrm{cmMTP}_2(\XX_B\mid\XX_A)$ is defined by
\begin{equation}
	\begin{split}
		&\mathbb P(\XX_A\in C,\XX_B\leq\xx_B)
		\mathbb P(\XX_A\in D,\XX_B\leq\yy_B)\\
		&\leq
		\mathbb P(\XX_A\in C\wedge D,
		\XX_B\leq\xx_B\wedge\yy_B)\\
		&\phantom{\leq}\times
		\mathbb P(\XX_A\in C\vee D,
		\XX_B\leq\xx_B\vee\yy_B),
	\end{split}
	\label{eq:intro-conditional-lattice}
\end{equation}
for admissible $C,D\in\mathcal B_A$ and
$\xx_B,\yy_B\in\Re^B$. Thus the conditioning
coordinates enter through measurable sets rather
than a density. We establish
$\mathrm{mMTP}_2(\XX)\Longrightarrow
\mathrm{cmMTP}_2(\XX_B\mid\XX_A)\Longrightarrow
\mathrm{MTP}_2(\XX)$, together with reduction of
the conditioning block and marginal stability.
The additional content of $\mathrm{cmMTP}_2$ is
not marginalization, which is already available
for classical total positivity, but the stronger
conditional restrictions in
\eqref{eq:intro-conditional-lattice}.
In particular, common monotone-factor models,
including comonotonic random vectors, satisfy
$\mathrm{cmMTP}_2$ even when their joint laws
are singular with respect to the product of
their marginals. Under the regularity conditions
of Theorem~\ref{thm:density-equivalence}, we obtain
$\mathrm{cmMTP}_2(\XX_B\mid\XX_A)
\Longleftrightarrow
\mathrm{cMTP}_2(\XX_B\mid\XX_A)$,
showing compatibility with the density-based
conditional formulation. We further establish
closure under $\mu_n\Rightarrow\mu$ and
coordinatewise non-decreasing Borel transformations,
including situations where
$\mu_n\ll\lambda^d$ and $\mu\perp\lambda^d$.
Consequently, conditional total positivity remains
available under singular approximation limits.

Our second contribution concerns structural
representations and tractable distributional
characterizations. For lattice-positive Markov
kernels, we establish
$K_1\in\mathcal K_{\mathrm{mLTP}_2}(A,B)$ and
$K_2\in\mathcal K_{\mathrm{LTP}_2}(B,C)$
$\Longrightarrow
K_1K_2\in\mathcal K_{\mathrm{LTP}_2}(A,C)$,
yielding $\mathrm{cmMTP}_2(\XX_C\mid\XX_A)$
under suitable Markov factorizations.
For continuous strictly increasing margins,
$\mathrm{cmMTP}_2$ is a copula property.
In particular, for a strict Archimedean copula
$C_\psi$ and $|A|=k$, we characterize
$\mathrm{cmMTP}_2(\UU_B\mid\UU_A)$ through
log-convexity of
$h_k=(-1)^k\psi^{(k)}$, under the stated
regularity conditions. This criterion applies
to the Clayton, Gumbel--Hougaard, Frank,
and Joe families and is preserved under
weak copula limits, including singular limits.
For nondegenerate Gaussian laws
$\XX\sim N_d(\boldsymbol\mu,\Sigma)$,
we obtain a complementary matrix characterization.
Writing $K=\Sigma^{-1}$, Corollary~
\ref{cor:gaussian-partial-correlation} establishes
\begin{equation}
	\mathrm{cmMTP}_2(\XX)
	\Longleftrightarrow
	K_{ij}\leq0,\ i\neq j
	\Longleftrightarrow
	\rho_{ij\mid[d]\setminus\{i,j\}}\geq0,\ i\neq j.
	\label{eq:intro-gaussian-characterization}
\end{equation}
Thus $\mathrm{cmMTP}_2(\XX)$ coincides with
$\mathrm{dMTP}_2(\XX)$ within the nondegenerate
Gaussian family, and verification reduces to
$d(d-1)/2$ precision-matrix sign restrictions.
Unlike nonnegative marginal correlations,
\eqref{eq:intro-gaussian-characterization}
requires nonnegative partial correlations,
providing a direct interpretation for Gaussian
graphical models and covariance selection
\citep{KarlinRinott1980a,FallatEtAl2017}.

Our third contribution develops multivariate
statistical inference for the conditional lattice
restrictions. Let
$\XX^{(1)},\ldots,\\ \XX^{(n)}
\stackrel{\mathrm{iid}}{\sim}\mu$
and fix a finite collection $\mathcal Q$
of $m$ admissible lattice comparisons.
The corresponding contrasts
$\boldsymbol\Delta=(\Delta_1,\ldots,\Delta_m)^\top$
satisfy $\boldsymbol\Delta\in[0,\infty)^m$
whenever $\mathrm{cmMTP}_2(\XX_B\mid\XX_A)$
holds. We construct empirical-frequency estimators
$\widehat{\boldsymbol\Delta}_n$ and establish
$\sqrt n(\widehat{\boldsymbol\Delta}_n-
\boldsymbol\Delta)\Rightarrow
N_m(\mathbf0,\Omega)$,
with a consistent plug-in estimator of $\Omega$.
This yields simultaneous tests of
$H_0:\boldsymbol\Delta\in[0,\infty)^m$
against
$H_1:\min_{r\in[m]}\Delta_r<0$,
with asymptotic size control, consistency
under fixed violations, and local-power
characterizations. The procedure requires
neither joint nor conditional density estimation
and remains applicable to discrete, continuous,
mixed, and singular laws. Importantly,
the finite system $\mathcal Q$ tests necessary
implications of Definition~\ref{def:conditional-measure-mtp2} rather than
the full collection of lattice inequalities.
Thus rejection provides evidence against
$\mathrm{cmMTP}_2$ and excludes incompatible
positive-dependence model restrictions,
including Gaussian $M$-matrix and suitable
Archimedean copula models, whereas
non-rejection does not establish
$\mathrm{cmMTP}_2$. The resulting inference
therefore provides a density-free specification
diagnostic for multivariate model assessment
rather than an omnibus verification of
conditional total positivity.

The paper is organized as follows.
Section~\ref{sec:conditional-mtp2} introduces
$\mathrm{mMTP}_2$ and $\mathrm{cmMTP}_2$,
establishes their implication structure,
common-factor criterion, conditioning-block
reduction, and marginal properties.
Section~\ref{sec:characterizations} proves
equivalence with density-based
$\mathrm{cMTP}_2$ under regularity conditions
and establishes preservation under weak
convergence and monotone transformations.
Section~\ref{sec:structural-representations}
develops lattice-positive Markov-kernel
composition and copula representations,
including Archimedean and Gaussian
characterizations and singular copula limits.
Section~\ref{sec:inference} establishes
the joint asymptotic distribution of empirical
lattice contrasts and develops simultaneous
testing, covariance estimation, consistency,
and local-power results.
Section~\ref{sec:discussion} discusses the
principal implications and directions for
further research.

\section{Conditional multivariate total positivity for probability measures}
\label{sec:conditional-mtp2}

This section develops a measure-theoretic formulation of conditional multivariate total positivity that does not require the existence of
densities. We introduce $\mathrm{mMTP}_2$ and $\mathrm{cmMTP}_2$, establish their basic marginal and reduction properties, and derive the
implication chain $\mathrm{mMTP}_2(\XX)\Longrightarrow
\mathrm{cmMTP}_2(\XX)\Longrightarrow
\mathrm{MTP}_2(\XX)$.

\subsection{Preliminaries and measure-theoretic formulation}
\label{subsec:preliminaries}

Throughout the paper, let $d\geq 2$ and write $[d]:=\{1,\ldots,d\}$.  For $\xx=(x_1,\ldots,x_d),\yy=(y_1,\ldots,y_d)\in\Re^d$, let
$\xx\wedge\yy:=(x_1\wedge y_1,\ldots,x_d\wedge y_d)$ and
$\xx\vee\yy:=(x_1\vee y_1,\ldots,x_d\vee y_d)$ denote the
componentwise minimum and maximum, respectively. The usual coordinatewise partial order on $\Re^d$ is denoted by $\xx\leq\yy$. For a non-empty set $A\subseteq[d]$, we define $\Re^A:=\prod_{i\in A}\Re$, $\xx_A:=(x_i)_{i\in A}$, and $\mathcal B_A:=\mathcal B(\Re^A)$, where $\mathcal B(\Re^A)$ is the Borel $\sigma$-algebra on $\Re^A$. If $A,B\subseteq[d]$ are
disjoint and non-empty, $\xx_{A\cup B}$ is identified with $(\xx_A,\xx_B)$ whenever no ambiguity can arise. For
$C,D\subseteq\Re^A$, define their lattice joins and meets by
\begin{equation}
	C\vee D:=\{\xx\vee\yy:\xx\in C,\ \yy\in D\},\qquad
	C\wedge D:=\{\xx\wedge\yy:\xx\in C,\ \yy\in D\}.
	\label{eq:set-lattice-operations}
\end{equation}
Unlike the corresponding pointwise operations, the sets in \eqref{eq:set-lattice-operations} need not be Borel for arbitrary $C,D\in\mathcal B_A$. Whenever probabilities involving $C\vee D$ or $C\wedge D$ are considered below, their measurability will therefore
be stated explicitly.

Let $\mu$ be a probability measure on $(\Re^d,\mathcal B_d)$ and let $\XX=(X_1,\ldots,X_d)$ be a random vector with distribution $\mu$. For $\varnothing\neq A\subseteq[d]$, denote the $A$-marginal of
$\mu$ by $\mu_A:=\mu\circ\pi_A^{-1}$, where
$\pi_A:\Re^d\to\Re^A$ is the canonical projection. If $A$ and $B$ form a partition of $[d]$, a regular conditional distribution of $\XX_B$ given $\XX_A$ is denoted by $K_{B\mid A}(\xx_A,\cdot)$. Therefore, for $E\in\mathcal B_A$ and $F\in\mathcal B_B$,
\begin{equation}
	\mu(E\times F)
	=
	\int_E K_{B\mid A}(\xx_A,F)\,\mu_A(d\xx_A).
	\label{eq:disintegration}
\end{equation}
Since $\Re^A$ and $\Re^B$ are standard Borel spaces, such a kernel exists and is $\mu_A$-almost surely unique \citep{Kallenberg2021}. Statements involving $K_{B\mid A}(\xx_A,\cdot)$ are understood with respect to a fixed version of the regular conditional distribution. A non-negative function $h:\Re^d\to[0,\infty)$ is multivariate
totally positive of order two, abbreviated $\mathrm{MTP}_2$, if $h(\xx)h(\yy)\leq h(\xx\wedge\yy)h(\xx\vee\yy)$ for all
$\xx,\yy\in\Re^d$. Total positivity in this sense is naturally associated with the lattice structure of the underlying state space and is stable under several operations of statistical interest,
including marginalization and, under suitable conditions,
conditioning \citep{KarlinRinott1980a,FallatEtAl2017}. The latter properties have made
$\mathrm{MTP}_2$ particularly useful in multivariate dependence and Markov structures \citep{FallatEtAl2017}. In the present paper the basic object is, however, a probability measure rather than a density. This distinction is essential because an $\mathrm{MTP}_2$
probability law need not admit a density with respect to a prescribed product measure. A density-free formulation of multivariate total positivity was developed by \citet{ColangeloEtAl2006}, in particular
to obtain a class of positively dependent probability measures that is closed under weak convergence. Our conditional formulation below is built on this measure-theoretic point of view.

We first recall the measure-theoretic counterpart of multivariate total positivity. The construction below does not require the existence of a density and will serve as the basic building block for the conditional
notion introduced in Section~\ref{subsec:definition}.

\begin{definition}[Measure-theoretic $\mathrm{MTP}_2$]
	\label{def:measure-mtp2}
	Let $\mu$ be a probability measure on $(\Re^d,\mathcal B_d)$. We say that $\mu$ is measure-theoretically multivariate totally positive of order two, and write $\mathrm{mMTP}_2(\mu)$, if
	\begin{equation}
		\mu(C)\mu(D)
		\leq
		\mu(C\wedge D)\mu(C\vee D)
		\label{eq:measure-mtp2}
	\end{equation}
	for all $C,D\in\mathcal B_d$ such that
	$C\wedge D,C\vee D\in\mathcal B_d$. A random vector $\XX$ is said to be $\mathrm{mMTP}_2$, and we write $\mathrm{mMTP}_2(\XX)$, if its distribution $\mathbb P_{\XX}$ satisfies $\mathrm{mMTP}_2(\mathbb P_{\XX})$.
\end{definition}

Definition~\ref{def:measure-mtp2} separates the lattice property of the probability measure from any particular density representation. The following remarks describes three features of this formulation that
will be used throughout: its restriction to lower orthants, its relation to the density-based setting, and the measurability of the set-valued lattice operations.

\begin{remark}
	\label{rem:measure-mtp2}
	\begin{enumerate}
		
		\noindent\item[(i)]
		For $\xx\in\Re^d$, define
		$L_{\xx}:=(-\infty,\xx]=\prod_{i=1}^d(-\infty,x_i]$.
		Then, for every $\xx,\yy\in\Re^d$,
		$L_{\xx}\wedge L_{\yy}=L_{\xx\wedge\yy}$ and
		$L_{\xx}\vee L_{\yy}=L_{\xx\vee\yy}$.
		Consequently, $\mathrm{mMTP}_2(\mu)$ and
		\eqref{eq:measure-mtp2} imply
		$F_{\mu}(\xx)F_{\mu}(\yy)
		\leq
		F_{\mu}(\xx\wedge\yy)F_{\mu}(\xx\vee\yy)$,
		where
		$F_{\mu}(\xx):=\mu(L_{\xx})$.
		Thus, viewed as the mapping
		$\xx\mapsto F_{\mu}(\xx)$,
		the distribution function of $\mu$ is $\mathrm{MTP}_2$. The converse does not follow directly from the definition, since
		$\mathrm{mMTP}_2(\mu)$ requires \eqref{eq:measure-mtp2} for all
		admissible pairs $C,D\in\mathcal B_d$, whereas the distribution function tests only the lower orthants $L_{\xx}$ and $L_{\yy}$.
		
		\item[(ii)]
		Let $\mu_i:=\mu\circ\pi_i^{-1}$, $i\in[d]$, and set $\nu:=\bigotimes_{i=1}^d\mu_i$. Definition~\ref{def:measure-mtp2} requires neither $\mu\ll\nu$ nor the existence of a Radon-Nikodym derivative
		$f:=d\mu/d\nu$. If $\mu\ll\nu$, the usual density-based condition is
		$f(\xx)f(\yy)\leq
		f(\xx\wedge\yy)f(\xx\vee\yy)$
		for $\xx,\yy\in\Re^d$, after choosing a suitable version of $f$. The relation between this pointwise property and
		$\mathrm{mMTP}_2(\mu)$ is not assumed here and will be established under appropriate conditions in Section~\ref{subsec:density-equivalence}.
		In particular, the present construction is meaningful when $\mu\perp\nu$, which is excluded by the density-based formulation in \citet{ColangeloEtAl2006}.
		
		\item[(iii)]
		Define
		$T_{\wedge}(\xx,\yy):=\xx\wedge\yy$ and
		$T_{\vee}(\xx,\yy):=\xx\vee\yy$ for
		$(\xx,\yy)\in\Re^d\times\Re^d$.
		Although $T_{\wedge}$ and $T_{\vee}$ are continuous, for
		$C,D\in\mathcal B_d$ the identities
		$C\wedge D=T_{\wedge}(C\times D)$ and
		$C\vee D=T_{\vee}(C\times D)$
		do not guarantee that $C\wedge D$ and $C\vee D$ are Borel. Accordingly, \eqref{eq:measure-mtp2} is imposed only when $C\wedge D,C\vee D\in\mathcal B_d$.The same measurability convention will be used for the conditional
		measure-theoretic formulation introduced in
		Section~\ref{subsec:definition}.
		
	\end{enumerate}
\end{remark}

\noindent The measure-theoretic approach also inherits a basic stability property that will be used repeatedly below. For $\varnothing\neq A\subseteq[d]$, recall that
$\mu_A=\mu\circ\pi_A^{-1}$ denotes the $A$-marginal of $\mu$.

\begin{lemma}[Marginal stability]
	\label{lem:mmtp2-marginal}
	If $\mathrm{mMTP}_2(\mu)$ holds, then
	$\mathrm{mMTP}_2(\mu_A)$ holds for every non-empty
	$A\subseteq[d]$.
\end{lemma}

\begin{proof}
	Fix $\varnothing\neq A\subseteq[d]$ and let
	$C,D\in\mathcal B_A$ be such that
	$C\wedge D,C\vee D\in\mathcal B_A$. Set
	$\widetilde C:=\pi_A^{-1}(C)$ and
	$\widetilde D:=\pi_A^{-1}(D)$. Then
	$\widetilde C,\widetilde D\in\mathcal B_d$ and, since the canonical
	projection preserves coordinatewise lattice operations,
	$\widetilde C\wedge\widetilde D
	=\pi_A^{-1}(C\wedge D)$ and
	$\widetilde C\vee\widetilde D
	=\pi_A^{-1}(C\vee D)$.
	Hence, by Definition~\ref{def:measure-mtp2},
	$\mu_A(C)\mu_A(D)
	\leq
	\mu_A(C\wedge D)\mu_A(C\vee D)$.
	Therefore $\mathrm{mMTP}_2(\mu_A)$ holds.
\end{proof}

\subsection{Conditional multivariate total positivity}
\label{subsec:definition}

We now move from the unconditional measure-theoretic property to the conditional setting. Let $A$ and $B$ be non-empty sets forming a partition of $[d]$, and write $\XX=(\XX_A,\XX_B)$. The density-based property $\mathrm{cMTP}_2(\XX_B\mid\XX_A)$ of 
\citet{FuchsWang2027} incorporates the conditioning coordinates through a density of $\XX_A$. Our objective is to retain this dependence on $\XX_A$ without imposing absolute continuity. For $\xx_B\in\Re^B$, define $L_{\xx_B}:=(-\infty,\xx_B]$. This leads to a
conditional lattice inequality formulated entirely in terms of the joint probability measure.

\begin{definition}[Conditional measure-theoretic $\mathrm{MTP}_2$]
	\label{def:conditional-measure-mtp2}
	Let $A$ and $B$ be non-empty sets forming a partition of $[d]$. We say that $\XX=(\XX_A,\XX_B)$ is conditionally measure-theoretically multivariate totally positive of order two, and
	write $\mathrm{cmMTP}_2(\XX_B\mid\XX_A)$, if
	\begin{equation}
		\begin{split}
			&\mathbb P(\XX_A\in C,\XX_B\in L_{\xx_B})
			\mathbb P(\XX_A\in D,\XX_B\in L_{\yy_B})\leq
			\mathbb P(\XX_A\in C\wedge D,
			\XX_B\in L_{\xx_B\wedge\yy_B})
			\mathbb P(\XX_A\in C\vee D,
			\XX_B\in L_{\xx_B\vee\yy_B})
		\end{split}
		\label{eq:cmmtp2}
	\end{equation}
	for all $\xx_B,\yy_B\in\Re^B$ and all $C,D\in\mathcal B_A$ such that $C\wedge D,C\vee D\in\mathcal B_A$.
\end{definition}

Definition~\ref{def:conditional-measure-mtp2} retains the conditioning coordinates through measurable subsets of $\Re^A$, rather than through a density of $\XX_A$. The following observations clarify how \eqref{eq:cmmtp2} relates the conditional formulation to the measure-theoretic and distribution-function notions of total positivity.

\begin{remark}
	\label{rem:conditional-measure-mtp2}
	\begin{enumerate}
		
		\noindent \item[(i)]
		Although the notation
		$\mathrm{cmMTP}_2(\XX_B\mid\XX_A)$ emphasizes the partition $(A,B)$, the property is formulated entirely in terms of the joint
		law $\mathbb P_{\XX}$. Indeed, if
		$\mu=\mathbb P_{\XX}$, then \eqref{eq:cmmtp2} is equivalent to
		\begin{equation}
			\begin{split}
				&\mu(C\times L_{\xx_B})\mu(D\times L_{\yy_B})\leq
				\mu((C\wedge D)\times L_{\xx_B\wedge\yy_B})
				\mu((C\vee D)\times L_{\xx_B\vee\yy_B}),
			\end{split}
			\label{eq:cmmtp2-joint-law}
		\end{equation}
		for every admissible $C,D\in\mathcal B_A$ and
		$\xx_B,\yy_B\in\Re^B$. More precisely,
		$\mathrm{cmMTP}_2(\XX_B\mid\XX_A)$ is invariant under the choice of a version of the regular conditional distribution $K_{B\mid A}$.
		
		\item[(ii)]
		For $\xx_A,\yy_A\in\Re^A$, set
		$C=L_{\xx_A}$ and $D=L_{\yy_A}$. Since,
		$L_{\xx_A}\wedge L_{\yy_A}=L_{\xx_A\wedge\yy_A}$, and
		$L_{\xx_A}\vee L_{\yy_A}=L_{\xx_A\vee\yy_A}$,
		\eqref{eq:cmmtp2} yields
		$F_{\XX}(\xx_A,\xx_B)F_{\XX}(\yy_A,\yy_B)
		\leq
		F_{\XX}(\xx_A\wedge\yy_A,\xx_B\wedge\yy_B)
		F_{\XX}(\xx_A\vee\yy_A,\xx_B\vee\yy_B)$. Hence,
		$\mathrm{cmMTP}_2(\XX_B\mid\XX_A)$ implies
		$\mathrm{MTP}_2(\XX)$ in the distribution-function sense.
		
		\item[(iii)]
		The conditioning coordinates are not passive in
		\eqref{eq:cmmtp2}. Let $\mathbf t_B=(t_i)_{i\in B}$ with $t_i\uparrow\infty$ for every $i\in B$. Then $L_{\mathbf t_B}\uparrow\Re^B$, and continuity from below implies
		$\mathbb P(\XX_A\in C,\XX_B\in L_{\mathbf t_B})
		\to\mu_A(C)$. Invoking \eqref{eq:cmmtp2} with
		$\xx_B=\yy_B=\mathbf t_B$ and passing to the limit yields $\mu_A(C)\mu_A(D)\leq
		\mu_A(C\wedge D)\mu_A(C\vee D)$.
		Therefore, $\mathrm{cmMTP}_2(\XX_B\mid\XX_A)$ implies $\mathrm{mMTP}_2(\XX_A)$.
		
	\end{enumerate}
\end{remark}

\noindent For $\xx_B\in\Re^B$, define the finite measure
$\mu_{\xx_B}^{A}$ on $(\Re^A,\mathcal B_A)$ by
$\mu_{\xx_B}^{A}(C):=
\mathbb P(\XX_A\in C,\XX_B\in L_{\xx_B})$.
The family $(\mu_{\xx_B}^{A})_{\xx_B\in\Re^B}$ is increasing in the
sense that $\xx_B\leq\yy_B$ implies
$\mu_{\xx_B}^{A}(C)\leq\mu_{\yy_B}^{A}(C)$ for every
$C\in\mathcal B_A$. Moreover,
$\mu_{\xx_B}^{A}(\Re^A)=F_{\XX_B}(\xx_B)$ and
$\mu_{\xx_B}^{A}(C)\uparrow\mu_A(C)$ as
$\xx_B\uparrow\boldsymbol{\infty}_B$. Proposition~\ref{prop:cmmtp2-finite-measures} provides an equivalent formulation of $\mathrm{cmMTP}_2(\XX_B\mid\XX_A)$ in terms of this family of finite measures.

\begin{proposition}
	\label{prop:cmmtp2-finite-measures}
	Let $A$ and $B$ be non-empty sets forming a partition of $[d]$. Then $\mathrm{cmMTP}_2(\XX_B\mid\XX_A)$ holds if and only if
	\begin{equation}
		\begin{split}
			\mu_{\xx_B}^{A}(C)\mu_{\yy_B}^{A}(D)
			\leq
			\mu_{\xx_B\wedge\yy_B}^{A}(C\wedge D)
			\mu_{\xx_B\vee\yy_B}^{A}(C\vee D)
		\end{split}
		\label{eq:cmmtp2-finite-measures}
	\end{equation}
	for all $\xx_B,\yy_B\in\Re^B$ and all
	$C,D\in\mathcal B_A$ such that
	$C\wedge D,C\vee D\in\mathcal B_A$.
\end{proposition}

\begin{proof}
	By the definition of $\mu_{\xx_B}^{A}$,
	$\mu_{\xx_B}^{A}(C)=
	\mathbb P(\XX_A\in C,\XX_B\in L_{\xx_B})$.
	Applying the same identity to
	$(\yy_B,D)$,
	$(\xx_B\wedge\yy_B,C\wedge D)$, and
	$(\xx_B\vee\yy_B,C\vee D)$ shows that
	\eqref{eq:cmmtp2-finite-measures} is precisely
	\eqref{eq:cmmtp2}. The two conditions are therefore equivalent.
\end{proof}

\noindent We first isolate a one-dimensional lattice property that will be used for distributions generated by a common ordered factor. Although
Definition~\ref{def:measure-mtp2} was introduced on $\Re^d$, the same set-based formulation applies on $\Re$.

\begin{lemma}[One-dimensional lattice inequality]
	\label{lem:one-dimensional-lattice}
	Let $\lambda$ be a probability measure on
	$(\Re,\mathcal B(\Re))$. Then, for all
	$E,F\in\mathcal B(\Re)$ such that
	$E\wedge F,E\vee F\in\mathcal B(\Re)$,
	\begin{equation}
		\lambda(E)\lambda(F)
		\leq
		\lambda(E\wedge F)\lambda(E\vee F).
		\label{eq:one-dimensional-lattice}
	\end{equation}
	In particular, every probability measure on $\Re$ is
	measure-theoretically $\mathrm{MTP}_2$.
\end{lemma}

\begin{proof}
	Equip $\Re$ with its usual total order. The constant function
	$h\equiv1$ is $\mathrm{MTP}_2$ with respect to $\lambda$, since
	$h(s)h(t)=h(s\wedge t)h(s\vee t)=1$ for all $s,t\in\Re$.
	The integral form of the basic composition inequality for
	$\mathrm{MTP}_2$ functions therefore gives
	$\lambda(E)\lambda(F)
	\leq
	\lambda(E\wedge F)\lambda(E\vee F)$
	for every admissible pair $E,F$. Hence the asserted
	measure-theoretic property follows.
\end{proof}

\noindent The preceding Lemma~\ref{lem:one-dimensional-lattice} yields a useful sufficient condition for
$\mathrm{cmMTP}_2$ that does not involve absolute continuity \citep{ellington2025metascorelens,pramanik2026optimal}. It covers, in particular, random vectors generated by monotone transformations of a common one-dimensional factor \citep{anderson2026obesity,pramanik2026stochastic,reed2026structural}.

\begin{proposition}[Common-factor criterion]
	\label{prop:common-factor}
	Let $Z$ be a real-valued random variable and let
	$\XX=(g_1(Z),\ldots,g_d(Z))$, where
	$g_i:\Re\to\Re$, $i\in[d]$, are non-decreasing Borel measurable functions \citep{dunbar2026modeling}. Then
	$\mathrm{cmMTP}_2(\XX_B\mid\XX_A)$ holds for every pair of non-empty sets $A,B$ forming a partition of $[d]$.
\end{proposition}

\begin{proof}
	Let $\lambda=\mathbb P_Z$ and fix a partition $(A,B)$. For $\xx_B,\yy_B\in\Re^B$ and admissible $C,D\in\mathcal B_A$, define
	$E:=\{z:g_A(z)\in C,\ g_B(z)\leq\xx_B\}$ and
	$F:=\{z:g_A(z)\in D,\ g_B(z)\leq\yy_B\}$ \citep{powell2025genomic,pramanik2024analysis}.
	Then $\mu_{\xx_B}^{A}(C)=\lambda(E)$, and
	$\mu_{\yy_B}^{A}(D)=\lambda(F)$ \citep{powell2026role,pramanik2026bayesian}. Since every $g_i$ is non-decreasing,
	$g_i(s\wedge t)=g_i(s)\wedge g_i(t)$ and
	$g_i(s\vee t)=g_i(s)\vee g_i(t)$ for all $s,t\in\Re$. Consequently, $E\wedge F$ is contained in $\{z:g_A(z)\in C\wedge D,\
	g_B(z)\leq\xx_B\wedge\yy_B\}$, where
	$E\vee F$ is contained in
	$\{z:g_A(z)\in C\vee D,\
	g_B(z)\leq\xx_B\vee\yy_B\}$. Lemma~\ref{lem:one-dimensional-lattice} therefore yields
	\begin{equation}
		\begin{split}
			\mu_{\xx_B}^{A}(C)\mu_{\yy_B}^{A}(D)
			&=\lambda(E)\lambda(F)\leq\lambda(E\wedge F)\lambda(E\vee F)\leq
			\mu_{\xx_B\wedge\yy_B}^{A}(C\wedge D)
			\mu_{\xx_B\vee\yy_B}^{A}(C\vee D).
		\end{split}
		\label{eq:common-factor-cmmtp2}
	\end{equation}
	The assertion follows from Proposition~\ref{prop:cmmtp2-finite-measures}.
\end{proof}

Proposition~\ref{prop:common-factor} shows that the
measure-theoretic conditional property includes distributions for which the density-based formulation need not be available. The canonical example is comonotonicity.

\begin{example}
	\label{ex:comonotonic}
	Let $U\sim\operatorname{Unif}(0,1)$ and let
	$F_1,\ldots,F_d$ be continuous and strictly increasing distribution functions. Define
	$\XX=(F_1^{-1}(U),\ldots,F_d^{-1}(U))$.
	Then $\XX$ is comonotonic and, by
	Proposition~\ref{prop:common-factor},
	$\mathrm{cmMTP}_2(\XX_B\mid\XX_A)$ holds for every pair of non-empty sets $A,B$ forming a partition of $[d]$. On the other hand, the law of $\XX$ is concentrated on the chain
	$\{(F_1^{-1}(u),\ldots,F_d^{-1}(u)):u\in(0,1)\}$.
	Since the marginal distributions are continuous \citep{pramanik2025dissecting,yusuf2025prognostic}, this set has measure zero under
	$\bigotimes_{i=1}^d\mathbb P_{X_i}$ whenever $d\geq2$. Consequently,
	$\mathbb P_{\XX}\perp
	\bigotimes_{i=1}^d\mathbb P_{X_i}$,
	so that $\XX$ does not admit the product-measure density required by the density-based formulation \citep{yusuf2025predictive}. Thus $\mathrm{cmMTP}_2$ applies to this extremal positively dependent model without imposing absolute
	continuity.
\end{example}

\noindent The common-factor construction illustrates that
$\mathrm{cmMTP}_2$ accommodates singular laws. We now show a structural property that will be used repeatedly below \citep{pramanik2026quantum}. In particular, the conditional measure-theoretic property is preserved when coordinates are transferred from the conditioning block to the conditioned block \citep{dasgupta2026frequent,ellington2025playmydata,kakkat2026angiotensin}.

\begin{theorem}[Reduction of the conditioning block]
	\label{thm:conditioning-reduction}
	Let $A$ and $B$ be non-empty sets forming a partition of $[d]$, with
	$|A|\geq2$. If $\mathrm{cmMTP}_2(\XX_B\mid\XX_A)$ holds, then, for every $k\in A$,
	$\mathrm{cmMTP}_2
	(\XX_{B\cup\{k\}}\mid\XX_{A\setminus\{k\}})$ holds.
\end{theorem}

\begin{proof}
	Fix $k\in A$ and set $A':=A\setminus\{k\}$ and
	$B':=B\cup\{k\}$. Let $C,D\in\mathcal B_{A'}$ be such that
	$C\wedge D,C\vee D\in\mathcal B_{A'}$, and let
	$\xx_{B'},\yy_{B'}\in\Re^{B'}$. Write
	$\xx_{B'}=(\xx_B,x_k)$ and
	$\yy_{B'}=(\yy_B,y_k)$, and define
	$C_x:=C\times(-\infty,x_k]$ and
	$D_y:=D\times(-\infty,y_k]$.
	Then $C_x,D_y\in\mathcal B_A$. We first identify the lattice operations on these sets. By the
	componentwise definition of $\wedge$ and $\vee$,
	\begin{equation}
		\begin{split}
			C_x\wedge D_y
			&=(C\wedge D)\times(-\infty,x_k\wedge y_k],\\
			C_x\vee D_y
			&=(C\vee D)\times(-\infty,x_k\vee y_k].
		\end{split}
		\label{eq:conditioning-reduction-sets}
	\end{equation}
	In particular, both sets in
	\eqref{eq:conditioning-reduction-sets} belong to $\mathcal B_A$ \citep{vikramdeo2024mitochondrial}. Hence, $C_x$ and $D_y$ form an admissible pair in
	Definition~\ref{def:conditional-measure-mtp2} \citep{pramanik2024estimation,pramanik2025optimal}. Next observe that
	$\{\XX_{A'}\in C,\XX_{B'}\in L_{\xx_{B'}}\}
	=
	\{\XX_A\in C_x,\XX_B\in L_{\xx_B}\}$,
	and analogously
	$\{\XX_{A'}\in D,\XX_{B'}\in L_{\yy_{B'}}\}
	=
	\{\XX_A\in D_y,\XX_B\in L_{\yy_B}\}$ \citep{valdez2025association,valdez2025exploring}. Since, $\mathrm{cmMTP}_2(\XX_B\mid\XX_A)$ holds,
	\eqref{eq:cmmtp2} applied to $C_x,D_y,\xx_B,\yy_B$ gives
	\begin{equation}
		\begin{split}
			&\mathbb P(\XX_A\in C_x,\XX_B\in L_{\xx_B})
			\mathbb P(\XX_A\in D_y,\XX_B\in L_{\yy_B})\leq
			\mathbb P(\XX_A\in C_x\wedge D_y,
			\XX_B\in L_{\xx_B\wedge\yy_B})\times
			\mathbb P(\XX_A\in C_x\vee D_y,
			\XX_B\in L_{\xx_B\vee\yy_B}).
		\end{split}
		\label{eq:conditioning-reduction-inequality}
	\end{equation}
Invoking \eqref{eq:conditioning-reduction-sets}, the first probability on the right-hand side of
	\eqref{eq:conditioning-reduction-inequality} equals
	$\mathbb P(\XX_{A'}\in C\wedge D,
	\XX_{B'}\in L_{\xx_{B'}\wedge\yy_{B'}})$,
	whereas the second equals
	$\mathbb P(\XX_{A'}\in C\vee D,
	\XX_{B'}\in L_{\xx_{B'}\vee\yy_{B'}})$. The two probabilities on the left-hand side are similarly
	$\mathbb P(\XX_{A'}\in C,\XX_{B'}\in L_{\xx_{B'}})$ \citep{pramanik2024stochastic}
	and
	$\mathbb P(\XX_{A'}\in D,\XX_{B'}\in L_{\yy_{B'}})$.
	Consequently, \eqref{eq:conditioning-reduction-inequality} is exactly the defining lattice inequality for
	$\mathrm{cmMTP}_2(\XX_{B'}\mid\XX_{A'})$. Since, $C,D$ and $\xx_{B'},\yy_{B'}$ were arbitrary, the assertion follows \citep{pramanik2025factors,pramanik2025strategies}.
\end{proof}

\noindent Repeated application of Theorem~\ref{thm:conditioning-reduction}
allows several coordinates to be removed from the conditioning block \citep{pramanik2025optimal,pramanik2025impact,pramanik2025strategic}.

\begin{corollary}
	\label{cor:iterated-conditioning-reduction}
	Let $A$ and $B$ be non-empty sets forming a partition of $[d]$, and suppose that
	$\mathrm{cmMTP}_2(\XX_B\mid\XX_A)$ holds. Then, for every non-empty
	$A'\subseteq A$, $\mathrm{cmMTP}_2
	(\XX_{B\cup(A\setminus A')}\mid\XX_{A'})$ holds.
\end{corollary}

\begin{proof}
	Enumerate $A\setminus A'=\{k_1,\ldots,k_m\}$. Successive application of Theorem~\ref{thm:conditioning-reduction} to
	$k_1,\ldots,k_m$ yields the assertion.
\end{proof}

\noindent Example~\ref{ex:singular-gaussian-factor} illustrates Theorem~\ref{thm:conditioning-reduction}
and Corollary~\ref{cor:iterated-conditioning-reduction} for a singular distribution and gives a direct numerical verification of the resulting
lattice inequality.

\begin{example}
	\label{ex:singular-gaussian-factor}
	Let $Z\sim\mathcal N(0,1)$ and define
	$\XX=(X_1,X_2,X_3):=(Z,2Z,3Z)$. Then $\XX$ is supported on the one-dimensional subspace
	$\{(z,2z,3z):z\in\Re\}$ and is therefore singular with respect to
	$\mathbb P_{X_1}\otimes\mathbb P_{X_2}\otimes\mathbb P_{X_3}$. Nevertheless, Proposition~\ref{prop:common-factor} gives
	$\mathrm{cmMTP}_2(X_3\mid(X_1,X_2))$ \citep{pramanik2023optimal,pramanik2024parametric,pramanik2024bayes}. Taking
	$A=\{1,2\}$, $B=\{3\}$, and $k=2$ in
	Theorem~\ref{thm:conditioning-reduction} therefore yields $\mathrm{cmMTP}_2((X_2,X_3)\mid X_1)$. To illustrate the latter property numerically, let
	$C=[-1,1/2]$, $D=[0,3/2]$,
	$\xx_B=(1,2)$, and $\yy_B=(2,1)$, where
	$B=\{2,3\}$ \citep{pramanik2023path1}. Since, $C\wedge D=[-1,1/2]$ and $C\vee D=[0,3/2]$, while $\xx_B\wedge\yy_B=(1,1)$ and
	$\xx_B\vee\yy_B=(2,2)$, we obtain
	\begin{equation}
		\begin{split}
			\mathbb P(X_1\in C,(X_2,X_3)\leq\xx_B)
			\mathbb P(X_1\in D,(X_2,X_3)\leq\yy_B)
			&=
			\mathbb P(-1\leq Z\leq1/2)
			\mathbb P(0\leq Z\leq1/3)\\
			&\approx
			(0.5328)(0.1306)
			=
			0.0696,
		\end{split}
		\label{eq:singular-gaussian-left}
	\end{equation}
	whereas
	\begin{equation}
		\begin{split}
			\mathbb P(X_1\in C\wedge D,
			(X_2,X_3)\leq\xx_B\wedge\yy_B)
			&\times
			\mathbb P(X_1\in C\vee D,
			(X_2,X_3)\leq\xx_B\vee\yy_B)\\
			&=
			\mathbb P(-1\leq Z\leq1/3)
			\mathbb P(0\leq Z\leq2/3)\approx
			(0.4719)(0.2475)
			=
			0.1168.
		\end{split}
		\label{eq:singular-gaussian-right}
	\end{equation}
	Hence, the lattice inequality is strict for this choice of $C,D,\xx_B,\yy_B$, since $0.0696<0.1168$.
	\end{example}

\noindent More generally, Corollary~\ref{cor:iterated-conditioning-reduction}
shows that the same common-factor model remains
$\mathrm{cmMTP}_2$ after any non-empty subset of the original conditioning coordinates is transferred to the conditioned block \citep{dasgupta2023frequent,dasgupta2026frequent,dunbar2026modeling}. Example~\ref{ex:singular-gaussian-factor} also emphasizes that this conclusion does not require a
density with respect to the product of the marginal laws \citep{pramanik2024estimation}.

\subsection{Basic implications and marginal properties}
\label{subsec:basic-properties}

We now pass from $\mathrm{cmMTP}_2(\XX_B\mid\XX_A)$ for a fixed partition $(A,B)$ of $[d]$ to the corresponding property of the random vector $\XX=(X_1,\ldots,X_d)$. Furthermore, we study the relations among $\mathrm{mMTP}_2(\XX)$, $\mathrm{cmMTP}_2(\XX_B\mid\XX_A)$, and
$\mathrm{MTP}_2(\XX)$, together with the behavior of
$\mathrm{cmMTP}_2$ under coordinate marginalization.

\begin{definition}
	\label{def:cmmtp2-vector}
	A $d$-dimensional random vector $\XX$ is said to be
	conditionally measure-theoretically multivariate totally positive of order two, and we write $\mathrm{cmMTP}_2(\XX)$, if
	$\mathrm{cmMTP}_2(\XX_B\mid\XX_A)$ holds for every pair of non-empty
	sets $A,B$ forming a partition of $[d]$.
\end{definition}

\noindent Therefore, $\mathrm{cmMTP}_2(\XX)$ requires
$\mathrm{cmMTP}_2(\XX_B\mid\XX_A)$ simultaneously over all$\varnothing\neq A\subsetneq[d]$, with $B=[d]\setminus A$. We next establish the marginal stability of $\mathrm{cmMTP}_2$. Theorem~\ref{thm:conditioning-reduction} permits
coordinates to be transferred from the conditioning block to the conditioned block \citep{khan2023myb,khan2024mp60}. Combined with continuity from below, this yields closure under marginalization in both blocks and, consequently, under arbitrary coordinate projections of $\XX$ \citep{pramanik2026dynamic}.

\begin{theorem}
	\label{thm:cmmtp2-marginalization}
	Let $A$ and $B$ be non-empty sets forming a partition of $[d]$. For a
	$d$-dimensional random vector $\XX$, we have:
	\begin{enumerate}
		\item[(i)]
		If $\mathrm{cmMTP}_2(\XX_B\mid\XX_A)$ holds, then
		$\mathrm{cmMTP}_2(\XX_{B'}\mid\XX_{A'})$ holds for all non-empty
		$A'\subseteq A$ and $B'\subseteq B$.
		
		\item[(ii)]
		If $\mathrm{cmMTP}_2(\XX)$ holds, then
		$\mathrm{cmMTP}_2(\XX_I)$ holds for every
		$I\subseteq[d]$ with $|I|\geq2$.
	\end{enumerate}
\end{theorem}

\begin{proof}
We first prove (i). Suppose that
$\mathrm{cmMTP}_2(\XX_B\mid\XX_A)$ holds. We begin with
marginalization in the conditioned block. The argument follows the usual marginalization principle for $\mathrm{MTP}_2$ measures, but is applied here to the mixed set-orthant formulation in
Definition~\ref{def:conditional-measure-mtp2}
\citep{ColangeloEtAl2006,FallatEtAl2017}.
Let $\varnothing\neq B'\subseteq B$ and set $J:=B\setminus B'$. Fix $\xx_{B'},\yy_{B'}\in\Re^{B'}$ and admissible
$C,D\in\mathcal B_A$. For $n\geq1$, define
$\xx_B^{(n)},\yy_B^{(n)}\in\Re^B$ by
$\xx_B^{(n)}=(\xx_{B'},n\boldsymbol{1}_J)$ and
$\yy_B^{(n)}=(\yy_{B'},n\boldsymbol{1}_J)$,
where $\boldsymbol{1}_J$ denotes the vector of ones indexed by $J$.
Definition~\ref{def:conditional-measure-mtp2} gives
	\begin{equation}
		\begin{split}
			&\mathbb P(\XX_A\in C,\XX_B\in L_{\xx_B^{(n)}})
			\mathbb P(\XX_A\in D,\XX_B\in L_{\yy_B^{(n)}})\\
			&\hspace{1cm}\leq
			\mathbb P(\XX_A\in C\wedge D,
			\XX_B\in L_{\xx_B^{(n)}\wedge\yy_B^{(n)}})
			\mathbb P(\XX_A\in C\vee D,
			\XX_B\in L_{\xx_B^{(n)}\vee\yy_B^{(n)}}).
		\end{split}
		\label{eq:conditioned-marginalization}
	\end{equation}
	Since $n\to\infty$, the four events in
	\eqref{eq:conditioned-marginalization} increase, respectively, to the
	corresponding events involving only $\XX_{B'}$. Continuity from below
	therefore yields
	\begin{equation}
		\begin{split}
			&\mathbb P(\XX_A\in C,\XX_{B'}\in L_{\xx_{B'}})
			\mathbb P(\XX_A\in D,\XX_{B'}\in L_{\yy_{B'}})\\
			&\hspace{1cm}\leq
			\mathbb P(\XX_A\in C\wedge D,
			\XX_{B'}\in L_{\xx_{B'}\wedge\yy_{B'}})
			\mathbb P(\XX_A\in C\vee D,
			\XX_{B'}\in L_{\xx_{B'}\vee\yy_{B'}}).
		\end{split}
		\label{eq:conditioned-marginal-limit}
	\end{equation}
	Hence, $\mathrm{cmMTP}_2(\XX_{B'}\mid\XX_A)$ holds. Now let $\varnothing\neq A'\subseteq A$ and set
	$K:=A\setminus A'$. By Corollary~\ref{cor:iterated-conditioning-reduction},
	$\mathrm{cmMTP}_2(\XX_{B\cup K}\mid\XX_{A'})$ holds \citep{issah2026stochastic,pramanik2026viscosity,pramanik2026stochastic}. Applying the conditioned-block marginalization just proved, first retaining $B$ and then $B'\subseteq
	B$, yields $\mathrm{cmMTP}_2(\XX_{B'}\mid\XX_{A'})$. This proves (i). We now prove (ii). Suppose $\mathrm{cmMTP}_2(\XX)$ holds and fix
	$I\subseteq[d]$ with $|I|\geq2$. Let $A'$ and $B'$ be arbitrary non-empty sets forming a partition of $I$. Define $A:=A'\cup I^c$ and $B:=B'$. Then $A$ and $B$ form a partition of
	$[d]$, and Definition~\ref{def:cmmtp2-vector} yields
	$\mathrm{cmMTP}_2(\XX_{B'}\mid\XX_A)$. Applying part (i) with $A'\subseteq A$ gives
	$\mathrm{cmMTP}_2(\XX_{B'}\mid\XX_{A'})$. Since the partition	$(A',B')$ of $I$ was arbitrary,
	$\mathrm{cmMTP}_2(\XX_I)$ follows from
	Definition~\ref{def:cmmtp2-vector}. This proves (ii).
\end{proof}

We now establish the basic implication structure connecting the three notions of total positivity considered above. The result is the
measure-theoretic \citep{issah2026bayesian,pramanik2026objective} counterpart of the implication chain for density-based conditional total positivity in \citet{FuchsWang2027}. In the present structure, however, the argument is formulated directly
at the level of the joint probability measure and therefore does not require absolute continuity or disintegration.

\begin{theorem}
	\label{thm:basic-implications}
	Let $A$ and $B$ be non-empty sets forming a partition of $[d]$. For a
	$d$-dimensional random vector $\XX$, we have:
	\begin{enumerate}
		\item[(i)]
		$\mathrm{mMTP}_2(\XX)$ implies
		$\mathrm{cmMTP}_2(\XX_B\mid\XX_A)$.
		
		\item[(ii)]
		$\mathrm{cmMTP}_2(\XX_B\mid\XX_A)$ implies
		$\mathrm{MTP}_2(\XX)$.
	\end{enumerate}
\end{theorem}

\begin{proof}
	We first prove (i). Let $\mu=\mathbb P_{\XX}$ and suppose that
	$\mathrm{mMTP}_2(\mu)$ holds. Fix
	$\xx_B,\yy_B\in\Re^B$ and
	$C,D\in\mathcal B_A$ such that
	$C\wedge D,C\vee D\in\mathcal B_A$. Define
	$E:=C\times L_{\xx_B}$ and
	$F:=D\times L_{\yy_B}$. Then $E,F\in\mathcal B_d$, and the
	componentwise lattice operations give
	\begin{equation}
		\begin{split}
			E\wedge F
			&=(C\wedge D)\times L_{\xx_B\wedge\yy_B},\
			\ E\vee F
			=(C\vee D)\times L_{\xx_B\vee\yy_B}.
		\end{split}
		\label{eq:product-set-lattice}
	\end{equation}
	Hence $E\wedge F,E\vee F\in\mathcal B_d$ \citep{pramanik2022lock,pramanik2022stochastic}. Applying
	Definition~\ref{def:measure-mtp2} to $E$ and $F$ yields
	\begin{equation}
		\begin{split}
			&\mathbb P(\XX_A\in C,\XX_B\in L_{\xx_B})
			\mathbb P(\XX_A\in D,\XX_B\in L_{\yy_B})\leq
			\mathbb P(\XX_A\in C\wedge D,
			\XX_B\in L_{\xx_B\wedge\yy_B})
			\mathbb P(\XX_A\in C\vee D,
			\XX_B\in L_{\xx_B\vee\yy_B}).
		\end{split}
		\label{eq:mmtp2-implies-cmmtp2}
	\end{equation}
	Since $C,D,\xx_B,\yy_B$ were arbitrary \citep{pramanik2024dependence,pramanik2024measuring},
	\eqref{eq:mmtp2-implies-cmmtp2} is precisely
	\eqref{eq:cmmtp2} \citep{pramanik2021consensus,pramanik2021optimization}. Therefore, $\mathrm{cmMTP}_2(\XX_B\mid\XX_A)$ holds. We now prove (ii). Suppose
	$\mathrm{cmMTP}_2(\XX_B\mid\XX_A)$ holds \citep{pramanik2016tail}. For
	$\xx=(\xx_A,\xx_B)$, and
	$\yy=(\yy_A,\yy_B)$ in $\Re^d$, take
	$C=L_{\xx_A}$ and $D=L_{\yy_A}$ in
	Definition~\ref{def:conditional-measure-mtp2}. Since,
	$L_{\xx_A}\wedge L_{\yy_A}=L_{\xx_A\wedge\yy_A}$ and
	$L_{\xx_A}\vee L_{\yy_A}=L_{\xx_A\vee\yy_A}$, we obtain $F_{\XX}(\xx)F_{\XX}(\yy)\leq
	F_{\XX}(\xx\wedge\yy)F_{\XX}(\xx\vee\yy)$.
	Hence, the mapping $\xx\mapsto F_{\XX}(\xx)$ is $\mathrm{MTP}_2$, which proves $\mathrm{MTP}_2(\XX)$.
\end{proof}

\noindent Theorem~\ref{thm:basic-implications} immediately yields the corresponding implication structure for the whole-vector property.

\begin{corollary}
	\label{cor:whole-vector-implications}
	For a $d$-dimensional random vector $\XX$,
	$\mathrm{mMTP}_2(\XX)$ implies
	$\mathrm{cmMTP}_2(\XX)$, and
	$\mathrm{cmMTP}_2(\XX)$ implies
	$\mathrm{MTP}_2(\XX)$.
\end{corollary}

\begin{proof}
	The first implication follows by applying
	Theorem~\ref{thm:basic-implications}(i) to every non-trivial partition $(A,B)$ of $[d]$ and invoking
	Definition~\ref{def:cmmtp2-vector}. The second implication follows from Theorem~\ref{thm:basic-implications}(ii).
\end{proof}

\noindent  The implication $\mathrm{mMTP}_2(\XX)\Longrightarrow \mathrm{cmMTP}_2(\XX_B\mid\XX_A)\Longrightarrow
\mathrm{MTP}_2(\XX)$
places conditional measure-theoretic total positivity between the measure-level and distribution-function notions. Its distinguishing feature is not marginal stability, which is already available for classical total positivity \citep{FallatEtAl2017,KarlinRinott1980a,reed2026modeling}, but the additional lattice restrictions on
$\mu^A_{\xx_B}(C)= \mathbb P(\XX_A\in C,\XX_B\leq\xx_B)$
for arbitrary admissible conditioning sets $C,D\in\mathcal B_A$. Unlike $\mathrm{MTP}_2(\XX)$, which constrains only joint lower-orthant probabilities \citep{hertweck2023clinicopathological,kakkat2026angiotensin}, $\mathrm{cmMTP}_2$ retains the dependence structure across conditioning events.
These restrictions remain meaningful for discrete, mixed \citep{pramanik2023path,pramanik2021optimal,pramanik2023scoring},
and singular laws, and provide the basis for the
density-free inferential processes of Section~\ref{sec:inference}.
Thus the contribution lies in conditional lattice
structure and its statistical implications, rather than
marginalization alone. Section~\ref{sec:characterizations} next examines its
relation to density-based conditional total positivity
and its stability under weak convergence.

\section{Characterizations and closure properties}
\label{sec:characterizations}

Section~\ref{sec:characterizations} establishes three structural properties of
$\mathrm{cmMTP}_2$: equivalence with $\mathrm{cMTP}_2$ under absolute-continuity conditions, closure under
$\mu_n\Rightarrow\mu$, and preservation under
$\YY=\phi(\XX)$ for coordinatewise non-decreasing $\phi$.
We begin with the density-based equivalence.

\subsection{Equivalence with the density-based formulation} \label{subsec:density-equivalence}

Let $A$ and $B$ be non-empty sets forming a partition of $[d]$ and define $\nu_A:=\bigotimes_{i\in A}\mathbb P_{X_i}$. Suppose $\mu_A\ll\nu_A$ and let $f_A:=d\mu_A/d\nu_A$. For a fixed version of
$K_{B\mid A}$, define, for $\xx_A\in\Re^A$ and
$\xx_B\in\Re^B$, $g_{\xx_B}(\xx_A):=
K_{B\mid A}(\xx_A,L_{\xx_B})f_A(\xx_A)$ \citep{hua2019assessing,kakkat2023cardiovascular,polansky2021motif}. Then disintegration gives
$\mu_{\xx_B}^{A}(C)=\int_C g_{\xx_B}(\xx_A)\nu_A(d\xx_A)$
for $C\in\mathcal B_A$. Therefore, the density-based property $\mathrm{cMTP}_2(\XX_B\mid\XX_A)$ of
\citet{FuchsWang2027} is the $\mathrm{MTP}_2$ property of
$(\xx_A,\xx_B)\mapsto g_{\xx_B}(\xx_A)$, whereas
$\mathrm{cmMTP}_2(\XX_B\mid\XX_A)$ is the corresponding lattice inequality for the finite measures
$(\mu_{\xx_B}^{A})_{\xx_B\in\Re^B}$. Lemma~\ref{lem:cmtp2-to-cmmtp2} gives the first connection between the two
formulations \citep{pramanik2025stubbornness,pramanik2026strategic,pramanik2024semicooperation}.

\begin{lemma}
	\label{lem:cmtp2-to-cmmtp2}
	Suppose that $\mu_A\ll\nu_A$ and that versions of $f_A$ and
	$K_{B\mid A}$ can be chosen such that
	$(\xx_A,\xx_B)\mapsto g_{\xx_B}(\xx_A)$ is
	$\mathrm{MTP}_2$. Then
	$\mathrm{cmMTP}_2(\XX_B\mid\XX_A)$ holds.
\end{lemma}

\begin{proof}
	Let $\xx_B,\yy_B\in\Re^B$ and let
	$C,D\in\mathcal B_A$ be such that
	$C\wedge D,C\vee D\in\mathcal B_A$. By disintegration, $\mu_{\xx_B}^{A}$ and $\mu_{\yy_B}^{A}$ are absolutely continuous
	with respect to $\nu_A$, with densities
	$g_{\xx_B}$ and $g_{\yy_B}$, respectively \citep{pramanik2023path}. Hence,
	$\mu_{\xx_B}^{A}(C)=\int_C g_{\xx_B}\,d\nu_A$ and
	$\mu_{\yy_B}^{A}(D)=\int_D g_{\yy_B}\,d\nu_A$ \citep{pramanik2021optimization,pramanik2020motivation,pramanik2024motivation}. Since $(\xx_A,\xx_B)\mapsto g_{\xx_B}(\xx_A)$ is
	$\mathrm{MTP}_2$, for $\uu_A,\vv_A\in\Re^A$ we have
	\begin{equation}
		\begin{split}
			g_{\xx_B}(\uu_A)g_{\yy_B}(\vv_A)
			\leq
			g_{\xx_B\wedge\yy_B}(\uu_A\wedge\vv_A)
			g_{\xx_B\vee\yy_B}(\uu_A\vee\vv_A).
		\end{split}
		\label{eq:density-set-inequality}
	\end{equation}
As $\nu_A=\bigotimes_{i\in A}\mathbb P_{X_i}$ is a product measure, the set-integral form of the $\mathrm{MTP}_2$ lattice inequality applied to \eqref{eq:density-set-inequality} implies
	\begin{equation}
		\begin{split}
			&\left(\int_C g_{\xx_B}\,d\nu_A\right)
			\left(\int_D g_{\yy_B}\,d\nu_A\right)\leq
			\left(\int_{C\wedge D}
			g_{\xx_B\wedge\yy_B}\,d\nu_A\right)
			\left(\int_{C\vee D}
			g_{\xx_B\vee\yy_B}\,d\nu_A\right).
		\end{split}
		\label{eq:density-to-measure-integrals}
	\end{equation}
	This is the set-based integration argument underlying the density-free formulation of $\mathrm{MTP}_2$ (see \citet[Theorem~3.10.14]{MullerStoyan2002} and
	\citet{ColangeloEtAl2006}). Invoking
	$\mu_{\zz_B}^{A}(E)=\int_Eg_{\zz_B}\,d\nu_A$
	for $E\in\mathcal B_A$ and $\zz_B\in\Re^B$,
	\eqref{eq:density-to-measure-integrals} becomes
	$\mu_{\xx_B}^{A}(C)\mu_{\yy_B}^{A}(D)
	\leq
	\mu_{\xx_B\wedge\yy_B}^{A}(C\wedge D)
	\mu_{\xx_B\vee\yy_B}^{A}(C\vee D)$.
	Proposition~\ref{prop:cmmtp2-finite-measures} now yields $\mathrm{cmMTP}_2(\XX_B\mid\XX_A)$. This completes the proof.
\end{proof}

\noindent The converse requires recovering the pointwise lattice inequality for $g$ from the set inequalities defining $\mathrm{cmMTP}_2$. We provide a localization argument under regularity conditions that ensure
convergence of the corresponding normalized probabilities \citep{pramanik2020optimization,pramanik2021effects,pramanik2021consensus,pramanik2023scoring}. Related density-based approaches of $\mathrm{MTP}_2$ are discussed in
\citet{FallatEtAl2017}, \citet{LauritzenEtAl2021}, and
\citet{RobevaEtAl2021}.

\begin{theorem} \label{thm:density-equivalence}
	Let $A$ and $B$ be non-empty sets forming a partition of $[d]$. Suppose, for every $i\in A$, $\mathbb P_{X_i}$ admits a strictly
	positive continuous Lebesgue density $p_i$, and define $\nu_A:=\bigotimes_{i\in A}\mathbb P_{X_i}$ \citep{pramanik2022lock,pramanik2023path1}. Assume  $\mu_A\ll\nu_A$ and that versions of $f_A$ and $K_{B\mid A}$ can be
	chosen such that $g(\xx_A,\xx_B):=K_{B\mid A}(\xx_A,L_{\xx_B})f_A(\xx_A)$ is continuous on $\Re^A\times\Re^B$. Then $\mathrm{cmMTP}_2(\XX_B\mid\XX_A)$ holds if and only if $\mathrm{cMTP}_2(\XX_B\mid\XX_A)$ holds.
\end{theorem}

\begin{proof}
	By Lemma~\ref{lem:cmtp2-to-cmmtp2}, it remains to prove that $\mathrm{cmMTP}_2(\XX_B\mid\XX_A)$ implies
	$\mathrm{cMTP}_2(\XX_B\mid\XX_A)$. Suppose $\mathrm{cmMTP}_2(\XX_B\mid\XX_A)$ holds and fix
	$\xx_A,\yy_A\in\Re^A$ and $\xx_B,\yy_B\in\Re^B$. For $r>0$, define
	the rectangles
	$C_r:=\prod_{i\in A}[x_i-r,x_i+r]$ and
	$D_r:=\prod_{i\in A}[y_i-r,y_i+r]$ \citep{hertweck2023clinicopathological,pramanik2022stochastic,pramanik2023optimal,pramanik2024optimization}.
	Then $C_r,D_r\in\mathcal B_A$ and
	\begin{equation}
		\begin{split}
			C_r\wedge D_r
			&=\prod_{i\in A}
			[x_i\wedge y_i-r,x_i\wedge y_i+r],\ \
			C_r\vee D_r
			=\prod_{i\in A}
			[x_i\vee y_i-r,x_i\vee y_i+r].
		\end{split}
		\label{eq:local-rectangles}
	\end{equation}
	Thus, $C_r\wedge D_r$ and $C_r\vee D_r$ are admissible in Definition~\ref{def:conditional-measure-mtp2}. Let $p_A(\zz_A):=\prod_{i\in A}p_i(z_i)$ \citep{pramanik2024estimation,pramanik2025construction}. Since
	$\nu_A(d\zz_A)=p_A(\zz_A)d\zz_A$, continuity of $g$ and $p_A$ implies, for every $\zz_A\in\Re^A$ and $\zz_B\in\Re^B$,
	\begin{equation}
		\begin{split}
			\lim_{r\downarrow0}
			\frac{\mu_{\zz_B}^{A}
				\left(\prod_{i\in A}[z_i-r,z_i+r]\right)}
			{(2r)^{|A|}}
			&=
			\lim_{r\downarrow0}
			\frac{1}{(2r)^{|A|}}
			\int_{\prod_{i\in A}[z_i-r,z_i+r]}
			g_{\zz_B}(\uu_A)p_A(\uu_A)\,d\uu_A\\
			&=
			g_{\zz_B}(\zz_A)p_A(\zz_A).
		\end{split}
		\label{eq:local-density-limit}
	\end{equation}
	Applying Proposition~\ref{prop:cmmtp2-finite-measures} to
	$C_r,D_r,\xx_B,\yy_B$ and invoking
	\eqref{eq:local-rectangles}, we obtain
	\begin{equation}
		\begin{split}
			\mu_{\xx_B}^{A}(C_r)\mu_{\yy_B}^{A}(D_r)
			\leq
			\mu_{\xx_B\wedge\yy_B}^{A}(C_r\wedge D_r)
			\mu_{\xx_B\vee\yy_B}^{A}(C_r\vee D_r).
		\end{split}
		\label{eq:localized-cmmtp2}
	\end{equation}
	Divide \eqref{eq:localized-cmmtp2} by $(2r)^{2|A|}$ and let $r\downarrow0$. By \eqref{eq:local-density-limit},
	\begin{equation}
		\begin{split}
			&g_{\xx_B}(\xx_A)g_{\yy_B}(\yy_A)
			p_A(\xx_A)p_A(\yy_A)\leq
			g_{\xx_B\wedge\yy_B}(\xx_A\wedge\yy_A)
			g_{\xx_B\vee\yy_B}(\xx_A\vee\yy_A)
			p_A(\xx_A\wedge\yy_A)p_A(\xx_A\vee\yy_A).
		\end{split}
		\label{eq:localized-density-inequality}
	\end{equation}
	Since $p_A$ is a product density,
	$p_A(\xx_A)p_A(\yy_A)= p_A(\xx_A\wedge\yy_A)p_A(\xx_A\vee\yy_A)$. Moreover, $p_A>0$ by assumption. Cancelling this common factor in \eqref{eq:localized-density-inequality} yields
	$g_{\xx_B}(\xx_A)g_{\yy_B}(\yy_A)
	\leq
	g_{\xx_B\wedge\yy_B}(\xx_A\wedge\yy_A)
	g_{\xx_B\vee\yy_B}(\xx_A\vee\yy_A)$. Since $\xx_A,\yy_A,\xx_B,\yy_B$ were arbitrary, $g$ is
	$\mathrm{MTP}_2$. Hence,
	$\mathrm{cMTP}_2(\XX_B\mid\XX_A)$ holds.
\end{proof}

\noindent Theorem~\ref{thm:density-equivalence} extends immediately to the whole-vector construction \citep{kakkat2026angiotensin,pramanik2026bayesian}.

\begin{corollary}
	\label{cor:density-equivalence-vector}
	Suppose every marginal $\mathbb P_{X_i}$, $i\in[d]$, admits a strictly positive continuous Lebesgue density. Assume, for every
	non-trivial partition $(A,B)$ of $[d]$, that $\mu_A\ll\nu_A$ and that
	versions of $f_A$ and $K_{B\mid A}$ can be chosen such that $(\xx_A,\xx_B)\mapsto
	K_{B\mid A}(\xx_A,L_{\xx_B})f_A(\xx_A)$
	is continuous \citep{pramanik2025optimal}. Then $\mathrm{cmMTP}_2(\XX)$ holds if and only if $\mathrm{cMTP}_2(\XX)$ holds \citep{pramanik2024estimation}.
\end{corollary}

\begin{proof}
	Apply Theorem~\ref{thm:density-equivalence} to every non-trivial
	partition $(A,B)$ of $[d]$ and use
	Definition~\ref{def:cmmtp2-vector}.
\end{proof}

\noindent Therefore, under the preceding regularity conditions, $\mathrm{cmMTP}_2$ agrees with the density-based $\mathrm{cMTP}_2$ property whenever the latter is available, while remaining meaningful for singular laws \citep{pramanik2023optimal,vikramdeo2024abstract,vikramdeo2024mitochondrial,vikramdeo2023profiling}. We next show that this extension also has the principal closure property motivating the measure theoretic formulation: stability under weak convergence.

\subsection{Weak convergence and approximation}
\label{subsec:weak-convergence}

We now study the behavior of $\mathrm{cmMTP}_2$ under weak convergence. This is a principal advantage of the measure theoretic construction: a sequence of absolutely continuous distributions may converge weakly to a singular law, so that a density-based conditional
property need not remain available at the limit. The corresponding weak-closure problem for unconditional measure-theoretic $\mathrm{MTP}_2$ was studied by \citet{ColangeloEtAl2006}. For a probability measure $\eta$ on $\Re^k$, let
$\mathcal C(\eta):=\{E\in\mathcal B(\Re^k):
\eta(\partial E)=0\}$ denote the class of $\eta$-continuity sets. If $\mu_n\Rightarrow\mu$, the Portmanteau theorem yields
$\mu_n(E)\to\mu(E)$ for every $E\in\mathcal C(\mu)$ (see \citet[Theorem~2.1]{Billingsley1999}). We use this
observation for the four sets appearing in the defining
$\mathrm{cmMTP}_2$ inequality.

\begin{lemma}[Determining class]
	\label{lem:cmmtp2-determining-class}
	Let $A$ and $B$ be non-empty sets forming a partition of $[d]$.
	Then $\mathrm{cmMTP}_2(\XX_B\mid\XX_A)$ holds if and only if
	\eqref{eq:cmmtp2} holds for all $\xx_B,\yy_B\in\Re^B$ and all
	strongly disjoint noncomparable closed rectangles $C,D\subset\Re^A$.
\end{lemma}

\begin{proof}
	Necessity is immediate. For the converse, let $\mathcal I_A$ be the
	class of half-open rectangles in $\Re^A$ and set
	$\mathcal A_A:=\operatorname{alg}(\mathcal I_A)$. For
	$C,D\in\mathcal A_A$, choose a common finite rectangular partition
	$(R_j)_{j\in J}$ such that
	$C=\dot\bigcup_{j\in J_C}R_j$ and
	$D=\dot\bigcup_{j\in J_D}R_j$. The assumed inequality on strongly disjoint noncomparable rectangles, together with the finite lattice inequality, yields
	\eqref{eq:cmmtp2} for $C,D\in\mathcal A_A$. Fix $\xx_B,\yy_B\in\Re^B$ and define
	$\alpha:=\mu_{\xx_B}^{A}$,
	$\beta:=\mu_{\yy_B}^{A}$,
	$\gamma:=\mu_{\xx_B\wedge\yy_B}^{A}$, and
	$\delta:=\mu_{\xx_B\vee\yy_B}^{A}$. For $D\in\mathcal A_A$, let
	$\mathcal G_D:=\{C\in\mathcal B_A:
	\alpha(C)\beta(D)\leq
	\gamma(C\wedge D)\delta(C\vee D)\}$.
	Then $\mathcal A_A\subseteq\mathcal G_D$, and continuity under monotone limits gives that $\mathcal G_D$ is a monotone class.
	Hence $\mathcal G_D=\mathcal B_A$. Similarly, for
	$C\in\mathcal B_A$, define
	$\mathcal H_C:=\{D\in\mathcal B_A:
	\alpha(C)\beta(D)\leq
	\gamma(C\wedge D)\delta(C\vee D)\}$.
	Again $\mathcal A_A\subseteq\mathcal H_C$ and
	$\mathcal H_C$ is a monotone class, so
	$\mathcal H_C=\mathcal B_A$. Proposition~\ref{prop:cmmtp2-finite-measures}
	now gives $\mathrm{cmMTP}_2(\XX_B\mid\XX_A)$.
\end{proof}

\begin{theorem} \label{thm:cmmtp2-weak-closure}
	Let $A$ and $B$ be non-empty sets forming a partition of $[d]$, and
	let $\mu_n,\mu$ be probability measures on $\Re^d$ such that $\mu_n\Rightarrow\mu$. Suppose
	$\mathrm{cmMTP}_2((\XX_n)_B\mid(\XX_n)_A)$ holds for every $n$, where $\XX_n\sim\mu_n$. Then
	$\mathrm{cmMTP}_2(\XX_B\mid\XX_A)$ holds for $\XX\sim\mu$.
\end{theorem}

\begin{proof}
	By Lemma~\ref{lem:cmmtp2-determining-class}, it is enough to verify \eqref{eq:cmmtp2} for strongly disjoint noncomparable closed rectangles $C,D\subset\Re^A$. Fix such $C,D$ and
	$\xx_B,\yy_B\in\Re^B$. Choose decreasing sequences of strongly disjoint noncomparable open
	rectangles $C_m\downarrow C$ and $D_m\downarrow D$, and $\xx_B^{(m)}\downarrow\xx_B$,
	$\yy_B^{(m)}\downarrow\yy_B$, such that
	\begin{equation}
		\begin{aligned}
			E_{1m}&:=C_m\times L_{\xx_B^{(m)}},&
			E_{2m}&:=D_m\times L_{\yy_B^{(m)}},\\
			E_{3m}&:=(C_m\wedge D_m)\times
			L_{\xx_B^{(m)}\wedge\yy_B^{(m)}},&
			E_{4m}&:=(C_m\vee D_m)\times
			L_{\xx_B^{(m)}\vee\yy_B^{(m)}}
		\end{aligned}
		\label{eq:weak-continuity-approximation}
	\end{equation}
	are $\mu$-continuity sets. Such sequences exist since the boundaries of rectangles are contained in finite unions of coordinate hyperplanes, and each marginal distribution has at most countably
	many atoms. For every $n,m\geq1$,
	$\mathrm{cmMTP}_2((\XX_n)_B\mid(\XX_n)_A)$ gives
	\begin{equation}
		\mu_n(E_{1m})\mu_n(E_{2m})
		\leq
		\mu_n(E_{3m})\mu_n(E_{4m}).
		\label{eq:weak-closure-nm}
	\end{equation}
	Fixing $m$ and letting $n\to\infty$, the Portmanteau theorem yields
	$\mu_n(E_{jm})\to\mu(E_{jm})$, $j\in[4]$, and hence,
	$\mu(E_{1m})\mu(E_{2m})
	\leq
	\mu(E_{3m})\mu(E_{4m})$. Finally, $C_m\downarrow C$, $D_m\downarrow D$ imply
	$C_m\wedge D_m\downarrow C\wedge D$ and
	$C_m\vee D_m\downarrow C\vee D$. Letting $m\to\infty$ and using
	continuity from above gives
	\begin{equation}
		\begin{split}
			&\mu(C\times L_{\xx_B})\mu(D\times L_{\yy_B})\leq
			\mu((C\wedge D)\times L_{\xx_B\wedge\yy_B})
			\mu((C\vee D)\times L_{\xx_B\vee\yy_B}).
		\end{split}
		\label{eq:weak-closure-limit}
	\end{equation}
	Lemma~\ref{lem:cmmtp2-determining-class} now yields
	$\mathrm{cmMTP}_2(\XX_B\mid\XX_A)$.
\end{proof}

\begin{remark}
	\label{rem:weak-closure}
	Theorem~\ref{thm:cmmtp2-weak-closure} requires neither $\mu_n\ll\lambda_d$ nor $\mu\ll\lambda_d$. Thus it may happen that
	$\mu_n\ll\lambda_d$ for every $n$, while
	$\mu_n\Rightarrow\mu$ and $\mu\perp\lambda_d$. More precisely, $\mu_n\Rightarrow\mu$ and
	$\mathrm{cmMTP}_2((\XX_n)_B\mid(\XX_n)_A)$ for every $n$ imply $\mathrm{cmMTP}_2(\XX_B\mid\XX_A)$ even when $\mu$ is singular. Moreover, under the hypotheses of Lemma~\ref{lem:cmtp2-to-cmmtp2},
	$\mathrm{cMTP}_2((\XX_n)_B\mid(\XX_n)_A)$ implies
	$\mathrm{cmMTP}_2((\XX_n)_B\mid(\XX_n)_A)$. Hence, the weak limit retains the measure theoretic conditional property even when a density-based approach is unavailable at $\mu$.
\end{remark}

\begin{example}
	\label{ex:weak-gaussian-limit}
	Let $Z,E_1,E_2$ be independent standard normal random variables and, for $\varepsilon>0$, define
	$X_{\varepsilon}:=Z+\varepsilon E_1$ and
	$Y_{\varepsilon}:=Z+\varepsilon E_2$. Then
	$(X_{\varepsilon},Y_{\varepsilon})\sim
	\mathcal N(\boldsymbol{0},\Sigma_{\varepsilon})$, where
	\begin{equation}
		\Sigma_{\varepsilon}
		=
		\begin{pmatrix}
			1+\varepsilon^2 & 1\\
			1 & 1+\varepsilon^2
		\end{pmatrix},
		\qquad
		\Sigma_{\varepsilon}^{-1}
		=
		\frac{1}{\varepsilon^2(2+\varepsilon^2)}
		\begin{pmatrix}
			1+\varepsilon^2 & -1\\
			-1 & 1+\varepsilon^2
		\end{pmatrix}.
		\label{eq:gaussian-weak-example}
	\end{equation}
	Hence, $\Sigma_{\varepsilon}^{-1}$ is an $M$-matrix and $(X_{\varepsilon},Y_{\varepsilon})$ is $\mathrm{dMTP}_2$ for every
	$\varepsilon>0$ \citep{FallatEtAl2017}. In particular, Theorem~\ref{thm:basic-implications}(i) yields $\mathrm{cmMTP}_2(Y_{\varepsilon}\mid X_{\varepsilon})$. Consider $\varepsilon\in\{1,1/2,1/5,1/10\}$. The
	correlations are
	\begin{center}
		\begin{tabular}{c|cccc}
			$\varepsilon$
			& $1$ & $1/2$ & $1/5$ & $1/10$\\ \hline
			$\operatorname{Corr}(X_{\varepsilon},Y_{\varepsilon})$
			& $0.5000$ & $0.8000$ & $0.9615$ & $0.9901$
		\end{tabular}
	\end{center}
	and therefore converge to $1$ as $\varepsilon\downarrow0$. The corresponding lower-orthant probabilities provide a direct
	probability-level illustration of the convergence. Since
	$\rho_{\varepsilon}:=
	\operatorname{Corr}(X_{\varepsilon},Y_{\varepsilon})
	=(1+\varepsilon^2)^{-1}$, the bivariate Gaussian orthant formula gives
	$\mathbb P(X_{\varepsilon}\leq0,Y_{\varepsilon}\leq0)
	=1/4+(2\pi)^{-1}\arcsin(\rho_{\varepsilon})$. Hence,
	
	\begin{center}
		\begin{tabular}{c|cccc}
			$\varepsilon$
			& $1$ & $1/2$ & $1/5$ & $1/10$\\ \hline
			$\mathbb P(X_{\varepsilon}\leq0,Y_{\varepsilon}\leq0)$
			& $0.3333$ & $0.3976$ & $0.4557$ & $0.4776$
		\end{tabular}
	\end{center}
	
	\noindent As $\varepsilon\downarrow0$, $\rho_{\varepsilon}\uparrow1$ and
	therefore,
	$\mathbb P(X_{\varepsilon}\leq0,Y_{\varepsilon}\leq0)\to1/2$.
	For the limiting vector $(Z,Z)$,
	$\mathbb P(Z\leq0,Z\leq0)=\mathbb P(Z\leq0)=1/2$.	Moreover,
	$(X_{\varepsilon},Y_{\varepsilon})\Rightarrow(Z,Z)$.
	The limiting distribution is concentrated on
	$\{(z,z):z\in\Re\}$, and hence is singular with respect to $\mathbb P_Z\otimes\mathbb P_Z$. Taking $A=\{1\}$ and $B=\{2\}$, Theorem~\ref{thm:cmmtp2-weak-closure} implies
	$\mathrm{cmMTP}_2(Z\mid Z)$ for the limiting random vector $(Z,Z)$. Therefore, every member of the approximating family admits a
	strictly positive Gaussian density, whereas the limit has no density with respect to the product of its marginals. Nevertheless, the conditional measure theoretic total positivity property is preserved.
\end{example}

\subsection{Monotone transformations}
\label{subsec:monotone-transformations}

We finally consider coordinatewise transformations
$\phi=(\phi_i)_{i\in[d]}$, where
$\phi_i:\Re\to\Re$ is increasing. Such maps satisfy
$\phi(\xx\wedge\yy)=\phi(\xx)\wedge\phi(\yy)$ and
$\phi(\xx\vee\yy)=\phi(\xx)\vee\phi(\yy)$, suggesting a natural closure property for $\mathrm{cmMTP}_2$. The analogous preservation
property for $\mathrm{MTP}_2$ is well known
\citep{FallatEtAl2017}; below we establish its measure theoretic conditional counterpart. For mappings $\phi_i:\Re\to\Re$, $i\in[d]$, write
$\phi_I(\xx_I):=(\phi_i(x_i))_{i\in I}$ for
$\emptyset\neq I\subseteq[d]$. The following elementary observation will be used repeatedly.

\begin{lemma}
	\label{lem:monotone-lattice-map}
	Suppose $\phi_i:\Re\to\Re$, $i\in I$, are strictly increasing Borel functions. Then, for $\xx_I,\yy_I\in\Re^I$,
	$\phi_I(\xx_I\wedge\yy_I)
	=\phi_I(\xx_I)\wedge\phi_I(\yy_I)$ and
	$\phi_I(\xx_I\vee\yy_I)
	=\phi_I(\xx_I)\vee\phi_I(\yy_I)$.
	Moreover, for admissible Borel sets
	$C,D\subseteq\phi_I(\Re^I)$,
	$\phi_I^{-1}(C\wedge D)
	=\phi_I^{-1}(C)\wedge\phi_I^{-1}(D)$ and
	$\phi_I^{-1}(C\vee D)
	=\phi_I^{-1}(C)\vee\phi_I^{-1}(D)$.
\end{lemma}

\begin{proof}
	For every $i\in I$, strict monotonicity gives
	$\phi_i(x_i\wedge y_i)
	=\phi_i(x_i)\wedge\phi_i(y_i)$ and
	$\phi_i(x_i\vee y_i)
	=\phi_i(x_i)\vee\phi_i(y_i)$.
	The vector identities follow coordinatewise. Applying them to
	$\xx_I\in\phi_I^{-1}(C)$ and
	$\yy_I\in\phi_I^{-1}(D)$, and using injectivity of $\phi_I$, gives the two inverse-image identities.
\end{proof}

\noindent Lemma~\ref{lem:monotone-lattice-map} allows the transformation property to be established separately on the conditioning and conditioned blocks.

\begin{lemma}
	\label{lem:lower-ray-extension}
	If $\mathrm{cmMTP}_2(\XX_B\mid\XX_A)$ holds, then its defining
	inequality remains valid when $L_{\xx_B}$ and $L_{\yy_B}$ are
	replaced by arbitrary coordinatewise lower rays, with intersection
	and union in the corresponding lattice terms.
\end{lemma}

\begin{proof}
	For lower rays $H_i\subseteq\Re$, choose
	$a_i^{(m)}$ such that
	$(-\infty,a_i^{(m)}]\uparrow H_i$. Hence, with
	$\aa_B^{(m)}:=(a_i^{(m)})_{i\in B}$,
	$L_{\aa_B^{(m)}}\uparrow\prod_{i\in B}H_i$.
	Apply Definition~\ref{def:conditional-measure-mtp2} and let
	$m\to\infty$. Continuity from below, together with
	$(H_i\cap G_i)$ and $(H_i\cup G_i)$ being lower rays, gives the
	assertion.
\end{proof}

\begin{proposition}
	\label{prop:monotone-transformations}
	Let $A$ and $B$ be non-empty sets forming a partition of $[d]$, and
	suppose $\mathrm{cmMTP}_2(\XX_B\mid\XX_A)$ holds.
	Let $\phi_i:\Re\to\Re$, $i\in[d]$, be strictly increasing Borel functions and define $Y_i:=\phi_i(X_i)$. Then
	\begin{enumerate}
		\item[(i)] $\mathrm{cmMTP}_2(\XX_B\mid\YY_A)$ holds.
		\item[(ii)] $\mathrm{cmMTP}_2(\YY_B\mid\YY_A)$ holds.
	\end{enumerate}
\end{proposition}

\begin{proof}
	We first prove (i). Write
	$\phi_A(\xx_A):=(\phi_i(x_i))_{i\in A}$ and let
	$C,D\in\mathcal B_A$ be admissible for $\YY_A$. Since
	$\YY_A\in\phi_A(\Re^A)$ almost surely, we may replace $C,D$ by
	$C\cap\phi_A(\Re^A)$ and $D\cap\phi_A(\Re^A)$, respectively.
	Set
	$C':=\phi_A^{-1}(C)$ and $D':=\phi_A^{-1}(D)$.
	By Lemma~\ref{lem:monotone-lattice-map},
	\begin{equation}
		\begin{split}
			C'\wedge D'&=\phi_A^{-1}(C\wedge D),\ \
			C'\vee D'=\phi_A^{-1}(C\vee D).
		\end{split}
		\label{eq:inverse-lattice-transform}
	\end{equation}
	Fix $\xx_B,\yy_B\in\Re^B$. Applying
	$\mathrm{cmMTP}_2(\XX_B\mid\XX_A)$ to
	$C',D',\xx_B,\yy_B$ and using
	$\{\YY_A\in E\}=\{\XX_A\in\phi_A^{-1}(E)\}$ gives
	\begin{equation}
		\begin{split}
			&\mathbb P(\YY_A\in C,\XX_B\in L_{\xx_B})
			\mathbb P(\YY_A\in D,\XX_B\in L_{\yy_B})\leq
			\mathbb P(\YY_A\in C\wedge D,
			\XX_B\in L_{\xx_B\wedge\yy_B})
			\mathbb P(\YY_A\in C\vee D,
			\XX_B\in L_{\xx_B\vee\yy_B}).
		\end{split}
		\label{eq:transform-conditioning-block}
	\end{equation}
	Therefore, $\mathrm{cmMTP}_2(\XX_B\mid\YY_A)$ holds.
	We now prove (ii). For $t\in\Re$ and $i\in B$, set
	$H_i(t):=\{s\in\Re:\phi_i(s)\leq t\}$. Since $\phi_i$ is strictly
	increasing, $H_i(t)$ is a lower ray and, for $s,t\in\Re$,
	$H_i(s\wedge t)=H_i(s)\cap H_i(t)$ and
	$H_i(s\vee t)=H_i(s)\cup H_i(t)$. Hence, for
	$\xx_B\in\Re^B$,
	\begin{equation}
		\{\YY_B\leq\xx_B\}
		=
		\{\XX_B\in H_B(\xx_B)\},
		\qquad
		H_B(\xx_B):=\prod_{i\in B}H_i(x_i).
		\label{eq:lower-ray-preimage}
	\end{equation}
	For admissible $C,D\in\mathcal B_A$, apply
	Lemma~\ref{lem:lower-ray-extension} to
	$H_B(\xx_B)$ and $H_B(\yy_B)$. Since
	$H_B(\xx_B\wedge\yy_B)
	=H_B(\xx_B)\cap H_B(\yy_B)$ and
	$H_B(\xx_B\vee\yy_B)
	=H_B(\xx_B)\cup H_B(\yy_B)$ coordinatewise, part (i) yields the
	defining inequality for
	$\mathrm{cmMTP}_2(\YY_B\mid\YY_A)$.
\end{proof}

\noindent Strict monotonicity in Proposition~\ref{prop:monotone-transformations}
can be removed by combining a strictly increasing perturbation with the weak-closure property of Theorem~\ref{thm:cmmtp2-weak-closure}.

\begin{theorem}
	\label{thm:monotone-transformations}
	Let $A$ and $B$ be non-empty sets forming a partition of $[d]$ and
	suppose $\mathrm{cmMTP}_2(\XX_B\mid\XX_A)$ holds. Let
	$\phi_i:\Re\to\Re$, $i\in[d]$, be non-decreasing Borel functions and
	define $Y_i:=\phi_i(X_i)$. Then
	$\mathrm{cmMTP}_2(\YY_B\mid\YY_A)$ holds.
\end{theorem}

\begin{proof}
	For $n\geq1$, define
	$\phi_i^{(n)}(x):=\phi_i(x)+x/n$ and
	$Y_i^{(n)}:=\phi_i^{(n)}(X_i)$. Each $\phi_i^{(n)}$ is strictly increasing, so Proposition~\ref{prop:monotone-transformations}(ii)
	gives $\mathrm{cmMTP}_2(\YY_B^{(n)}\mid\YY_A^{(n)})$.
	Moreover, $\YY^{(n)}\to\YY$ almost surely and hence,
	$\YY^{(n)}\Rightarrow\YY$. Theorem~\ref{thm:cmmtp2-weak-closure}
	therefore yields $\mathrm{cmMTP}_2(\YY_B\mid\YY_A)$.
\end{proof}

\begin{corollary}
	\label{cor:monotone-transformations-vector}
	Suppose $\mathrm{cmMTP}_2(\XX)$ holds and let
	$\phi_i:\Re\to\Re$, $i\in[d]$, be non-decreasing Borel functions. Then
	$\phi(\XX):=(\phi_1(X_1),\ldots,\phi_d(X_d))$
	is $\mathrm{cmMTP}_2$.
\end{corollary}

\begin{proof}
	For every non-trivial partition $(A,B)$ of $[d]$,
	Definition~\ref{def:cmmtp2-vector} and
	Theorem~\ref{thm:monotone-transformations} give
	$\mathrm{cmMTP}_2(\phi_B(\XX_B)\mid\phi_A(\XX_A))$.
	Hence $\mathrm{cmMTP}_2(\phi(\XX))$.
\end{proof}

\begin{example}
	\label{ex:monotone-thresholding}
	Let $Z\sim N(0,1)$ and
	$\XX=(Z,2Z,3Z)$. By Proposition~\ref{prop:common-factor},
	$\mathrm{cmMTP}_2(\XX)$ holds. Define
	$\phi_i(x):=\mathbf 1_{\{x>0\}}$, $i\in[3]$, and
	$\YY:=\phi(\XX)$. Then $\YY=(I,I,I)$, where
	$I:=\mathbf 1_{\{Z>0\}}\sim\operatorname{Bernoulli}(1/2)$.
	Since each $\phi_i$ is non-decreasing,
	Corollary~\ref{cor:monotone-transformations-vector} yields
	$\mathrm{cmMTP}_2(\YY)$. For $A=\{1,2\}$, $B=\{3\}$,
	$C=\{(0,0)\}$, $D=\{(1,1)\}$, $x_3=0$, and $y_3=1$,
	the two sides of \eqref{eq:cmmtp2} are both
	$(1/2)(1/2)=1/4$.
\end{example}

\section{Structural representations of conditional total positivity} \label{sec:structural-representations}

Section~\ref{sec:structural-representations} develops two structural representations of $\mathrm{cmMTP}_2$. First, for
$K_1\in\mathcal K_{\mathrm{mLTP}_2}(A,B)$ and
$K_2\in\mathcal K_{\mathrm{LTP}_2}(B,C)$, we establish
$K_1K_2\in\mathcal K_{\mathrm{LTP}_2}(A,C)$ and derive
$\mathrm{cmMTP}_2(\XX_C\mid\XX_A)$ under the Markov factorization
$\lambda(dx_A)K_1(x_A,dx_B)K_2(x_B,dx_C)$. Second, at the copula level we obtain
$\mathrm{cmMTP}_2(\XX_B\mid\XX_A)
\Longleftrightarrow
\mathrm{cmMTP}_2(\UU_B\mid\UU_A)$
under continuous strictly increasing margins, characterize the Archimedean case through log-convexity of $h_k=(-1)^k\psi^{(k)}$, and prove the weak-closure principle $C_n\to C,\ \mathrm{cmMTP}_2(\UU^{(n)})
\Rightarrow\mathrm{cmMTP}_2(\UU)$, including singular copula limits.

\subsection{Positive kernels and Markov products}
\label{subsec:positive-kernels-markov-products}

Let $A,B,C$ be pairwise disjoint non-empty index sets. Let $\lambda$ be a probability measure on $(\Re^A,\mathcal B_A)$ and let
$K_1:\Re^A\times\mathcal B_B\to[0,1]$, and
$K_2:\Re^B\times\mathcal B_C\to[0,1]$ be Markov kernels. Their composition is the kernel $K_1K_2:\Re^A\times\mathcal B_C\to[0,1]$
defined by
\begin{equation}
	(K_1K_2)(\xx_A,E)
	:=
	\int_{\Re^B}K_2(\xx_B,E)K_1(\xx_A,d\xx_B),
	\qquad E\in\mathcal B_C.
	\label{eq:kernel-composition}
\end{equation}
The composition principle is classical in total positivity \citep{KarlinRinott1980a}. Here we formulate a measure theoretic version adapted to the set-orthant structure of $\mathrm{cmMTP}_2$. For a Markov kernel
$K:\Re^A\times\mathcal B_B\to[0,1]$ and
$\xx_B\in\Re^B$, define
$K_{\xx_B}(\xx_A):=K(\xx_A,L_{\xx_B})$.
We call $K$ lattice-positive of order two if
\begin{equation}
	K_{\xx_B}(\xx_A)K_{\yy_B}(\yy_A)
	\leq
	K_{\xx_B\wedge\yy_B}(\xx_A\wedge\yy_A)
	K_{\xx_B\vee\yy_B}(\xx_A\vee\yy_A)
	\label{eq:lattice-positive-kernel}
\end{equation}
for all $\xx_A,\yy_A\in\Re^A$ and
$\xx_B,\yy_B\in\Re^B$. The class of kernels satisfying
\eqref{eq:lattice-positive-kernel} is denoted by
$\mathcal K_{\mathrm{LTP}_2}(A,B)$. For composition, we use its measure theoretic counterpart. A kernel
$K:\Re^A\times\mathcal B_B\to[0,1]$ belongs to
$\mathcal K_{\mathrm{mLTP}_2}(A,B)$ if
\begin{equation}
	K(\xx_A,C)K(\yy_A,D)
	\leq
	K(\xx_A\wedge\yy_A,C\wedge D)
	K(\xx_A\vee\yy_A,C\vee D)
	\label{eq:measure-lattice-positive-kernel}
\end{equation}
for all $\xx_A,\yy_A\in\Re^A$ and all
$C,D\in\mathcal B_B$ such that
$C\wedge D,C\vee D\in\mathcal B_B$. Clearly,
$\mathcal K_{\mathrm{mLTP}_2}(A,B)
\subseteq\mathcal K_{\mathrm{LTP}_2}(A,B)$. The following integral form of the lattice inequality is the
basic tool for kernel composition.

\begin{lemma}[Lattice-positive kernel integration]
	\label{lem:kernel-integration}
	Let $K\in\mathcal K_{\mathrm{mLTP}_2}(A,B)$ and let
	$H:\Re^B\times\mathcal B_C\to[0,1]$ satisfy
	$H\in\mathcal K_{\mathrm{LTP}_2}(B,C)$. Then, for
	$\xx_A,\yy_A\in\Re^A$ and $\xx_C,\yy_C\in\Re^C$,
	\begin{equation}
		\begin{split}
			\int H(\uu_B,L_{\xx_C})K(\xx_A,d\uu_B)
			\int H(\vv_B,L_{\yy_C})K(\yy_A,d\vv_B)&\leq
			\int H(\ww_B,L_{\xx_C\wedge\yy_C})
			K(\xx_A\wedge\yy_A,d\ww_B)\\
			&\phantom{\leq}\times
			\int H(\zz_B,L_{\xx_C\vee\yy_C})
			K(\xx_A\vee\yy_A,d\zz_B).
		\end{split}
		\label{eq:kernel-integration}
	\end{equation}
\end{lemma}

\begin{proof}
	Put
	$f_1(\uu_B):=H(\uu_B,L_{\xx_C})$,
	$f_2(\vv_B):=H(\vv_B,L_{\yy_C})$,
	$f_3(\ww_B):=H(\ww_B,L_{\xx_C\wedge\yy_C})$, and
	$f_4(\zz_B):=H(\zz_B,L_{\xx_C\vee\yy_C})$.
	Since $H\in\mathcal K_{\mathrm{LTP}_2}(B,C)$,
	$f_1(\uu_B)f_2(\vv_B)
	\leq f_3(\uu_B\wedge\vv_B)f_4(\uu_B\vee\vv_B)$.
	Set
	$K_1:=K(\xx_A,\cdot)$,
	$K_2:=K(\yy_A,\cdot)$,
	$K_3:=K(\xx_A\wedge\yy_A,\cdot)$, and
	$K_4:=K(\xx_A\vee\yy_A,\cdot)$.
	By \eqref{eq:measure-lattice-positive-kernel},
	$K_1(E)K_2(F)\leq K_3(E\wedge F)K_4(E\vee F)$ for every admissible $E,F\in\mathcal B_B$. The integral four-functions inequality \citep{KarlinRinott1980a} now yields $(\int f_1\,dK_1)(\int f_2\,dK_2)
	\leq(\int f_3\,dK_3)(\int f_4\,dK_4)$, which is
	\eqref{eq:kernel-integration}. This completes the proof.
\end{proof}

\noindent Lemma~\ref{lem:kernel-integration} yields the following composition property for lattice-positive kernels.

\begin{proposition}[Kernel composition]
	\label{prop:kernel-composition}
	Let $K\in\mathcal K_{\mathrm{mLTP}_2}(A,B)$ and
	$H\in\mathcal K_{\mathrm{LTP}_2}(B,C)$. Then
	$KH\in\mathcal K_{\mathrm{LTP}_2}(A,C)$.
\end{proposition}

\begin{proof}
	For $\xx_A,\yy_A\in\Re^A$ and
	$\xx_C,\yy_C\in\Re^C$, \eqref{eq:kernel-composition} gives
	$(KH)_{\xx_C}(\xx_A)
	=\int H(\uu_B,L_{\xx_C})K(\xx_A,d\uu_B)$ and analogously for
	$(\yy_A,\yy_C)$. Hence Lemma~\ref{lem:kernel-integration} yields
	\begin{equation}
		\begin{split}
			&(KH)_{\xx_C}(\xx_A)(KH)_{\yy_C}(\yy_A)\leq
			(KH)_{\xx_C\wedge\yy_C}(\xx_A\wedge\yy_A)
			(KH)_{\xx_C\vee\yy_C}(\xx_A\vee\yy_A).
		\end{split}
		\label{eq:kernel-composition-ltp2}
	\end{equation}
	Therefore, $KH\in\mathcal K_{\mathrm{LTP}_2}(A,C)$.
\end{proof}

\begin{remark}[Conditional independence and positive kernels]
	\label{rem:conditional-independence-kernels}
	Let $(\XX_A,\XX_B,\XX_C)$ have distribution
	\begin{equation}
		\mathbb P(d\xx_A,d\xx_B,d\xx_C)
		=
		\lambda(d\xx_A)
		K_1(\xx_A,d\xx_B)
		K_2(\xx_B,d\xx_C).
		\label{eq:markov-factorization}
	\end{equation}
	Then $\XX_A\perp\XX_C\mid\XX_B$, and the conditional kernel of
	$\XX_C$ given $\XX_A=\xx_A$ is
	\begin{equation}
		K_{C\mid A}(\xx_A,E)
		=
		(K_1K_2)(\xx_A,E)
		=
		\int_{\Re^B}K_2(\xx_B,E)K_1(\xx_A,d\xx_B).
		\label{eq:conditional-markov-kernel}
	\end{equation}
	While $\XX_A\perp\XX_C\mid\XX_B$ yields
	\eqref{eq:markov-factorization}, it does not by itself imply $K_{C\mid A}\in\mathcal K_{\mathrm{LTP}_2}(A,C)$. By
	Proposition~\ref{prop:kernel-composition},
	$K_1\in\mathcal K_{\mathrm{mLTP}_2}(A,B)$ and
	$K_2\in\mathcal K_{\mathrm{LTP}_2}(B,C)$ imply
	$K_{C\mid A}=K_1K_2\in\mathcal K_{\mathrm{LTP}_2}(A,C)$. Thus, for multivariate $\XX_B$, kernel positivity supplements
	$\XX_A\perp\XX_C\mid\XX_B$ compared with \citet{FuchsWang2027}.
\end{remark}

\begin{theorem} \label{thm:markov-product}
	Let $A,B,C$ be pairwise disjoint non-empty index sets and suppose
	\begin{equation}
		\mathbb P(d\xx_A,d\xx_B,d\xx_C)
		=
		\lambda(d\xx_A)K_1(\xx_A,d\xx_B)K_2(\xx_B,d\xx_C),
		\label{eq:markov-product-law}
	\end{equation}
	where $\lambda$ is $\mathrm{mMTP}_2$,
	$K_1\in\mathcal K_{\mathrm{mLTP}_2}(A,B)$, and
	$K_2\in\mathcal K_{\mathrm{LTP}_2}(B,C)$. Then
	\begin{enumerate}
		\item[(i)]
		$K_{C\mid A}=K_1K_2\in
		\mathcal K_{\mathrm{LTP}_2}(A,C)$.
		\item[(ii)]
		$\mathrm{cmMTP}_2(\XX_C\mid\XX_A)$ holds.
	\end{enumerate}
\end{theorem}

\begin{proof}
	We first prove (i). By \eqref{eq:markov-product-law},
	$\XX_A\perp\XX_C\mid\XX_B$, and disintegration gives, for $E\in\mathcal B_C$, $K_{C\mid A}(\xx_A,E)
	=\int_{\Re^B}K_2(\xx_B,E)\\ K_1(\xx_A,d\xx_B)
	=(K_1K_2)(\xx_A,E)$. Since,
	$K_1\in\mathcal K_{\mathrm{mLTP}_2}(A,B)$ and
	$K_2\in\mathcal K_{\mathrm{LTP}_2}(B,C)$,
	Proposition~\ref{prop:kernel-composition} yields
	$K_{C\mid A}\in\mathcal K_{\mathrm{LTP}_2}(A,C)$.
	We now prove (ii). For $\xx_C\in\Re^C$ and
	$E\in\mathcal B_A$, define
	$\eta_{\xx_C}^{A}(E):=
	\mathbb P(\XX_A\in E,\XX_C\in L_{\xx_C})$. By (i),
	\begin{equation}
		\eta_{\xx_C}^{A}(E)
		=
		\int_EK_{C\mid A}(\xx_A,L_{\xx_C})\lambda(d\xx_A).
		\label{eq:markov-product-finite-measure}
	\end{equation}
	Fix $\xx_C,\yy_C\in\Re^C$ and admissible
	$E,F\in\mathcal B_A$. Since
	$K_{C\mid A}\in\mathcal K_{\mathrm{LTP}_2}(A,C)$,
	\begin{equation}
		\begin{split}
			&K_{C\mid A}(\xx_A,L_{\xx_C})
			K_{C\mid A}(\yy_A,L_{\yy_C})\leq
			K_{C\mid A}(\xx_A\wedge\yy_A,L_{\xx_C\wedge\yy_C})
			K_{C\mid A}(\xx_A\vee\yy_A,L_{\xx_C\vee\yy_C}).
		\end{split}
		\label{eq:markov-product-pointwise}
	\end{equation}
	As $\lambda$ is $\mathrm{mMTP}_2$, the integral
	four-functions inequality applied to
	\eqref{eq:markov-product-pointwise} yields
	\begin{equation}
		\begin{split}
			\eta_{\xx_C}^{A}(E)\eta_{\yy_C}^{A}(F)
			&\leq
			\eta_{\xx_C\wedge\yy_C}^{A}(E\wedge F)
			\eta_{\xx_C\vee\yy_C}^{A}(E\vee F).
		\end{split}
		\label{eq:markov-product-cmmtp2}
	\end{equation}
	Proposition~\ref{prop:cmmtp2-finite-measures} therefore yields $\mathrm{cmMTP}_2(\XX_C\mid\XX_A)$.
\end{proof}

\noindent Theorem~\ref{thm:markov-product} can be illustrated explicitly for finite-state kernels. The following binary Markov chain shows how lattice positivity is propagated through kernel composition and
yields $\mathrm{cmMTP}_2$ for the endpoint variables.

\begin{example}[Binary Markov product]
	\label{ex:binary-markov-product}
	Let $(X_1,X_2,X_3)\in\{0,1\}^3$ form the Markov chain
	$X_1-X_2-X_3$, with
	$\mathbb P(X_1=0)=\mathbb P(X_1=1)=1/2$ and transition kernels
	\begin{center}
		\begin{tabular}{c|cc}
			& $X_2=0$ & $X_2=1$\\ \hline
			$X_1=0$ & $0.8$ & $0.2$\\
			$X_1=1$ & $0.3$ & $0.7$
		\end{tabular}
		\qquad
		\begin{tabular}{c|cc}
			& $X_3=0$ & $X_3=1$\\ \hline
			$X_2=0$ & $0.9$ & $0.1$\\
			$X_2=1$ & $0.4$ & $0.6$
		\end{tabular}.
	\end{center}
	Take $A=\{1\}$, $B=\{2\}$, and $C=\{3\}$.
	
	\textit{(i)}
	Let $K_1$ and $K_2$ denote the two kernels above. Since the state
	spaces are binary,
	$
	K_1(0,0)K_1(1,1)=0.56>0.06
	=K_1(0,1)K_1(1,0),
	$
	and
	$
	K_2(0,0)K_2(1,1)=0.54>0.04
	=K_2(0,1)K_2(1,0).
	$
	Therefore, $K_1$ and $K_2$ satisfy the corresponding lattice-positive inequalities. In particular,
	$K_1\in\mathcal K_{\mathrm{mLTP}_2}(A,B)$ and
	$K_2\in\mathcal K_{\mathrm{LTP}_2}(B,C)$.
	
	\textit{(ii)}
	By \eqref{eq:kernel-composition}, the Markov product is
	\begin{equation}
		K_1K_2
		=
		\begin{pmatrix}
			0.8&0.2\\
			0.3&0.7
		\end{pmatrix}
		\begin{pmatrix}
			0.9&0.1\\
			0.4&0.6
		\end{pmatrix}
		=
		\begin{pmatrix}
			0.80&0.20\\
			0.55&0.45
		\end{pmatrix}.
		\label{eq:binary-markov-product}
	\end{equation}
	Hence,
	$\mathbb P(X_3=0\mid X_1=0)=0.80$ and
	$\mathbb P(X_3=0\mid X_1=1)=0.55$, while
	$\mathbb P(X_3\leq1\mid X_1=i)=1$, $i\in\{0,1\}$.
	Moreover,
	$0.80(0.45)=0.36>0.11=0.20(0.55)$, so that
	$K_1K_2\in\mathcal K_{\mathrm{LTP}_2}(A,C)$, in agreement with
	Proposition~\ref{prop:kernel-composition}. Finally, since $\lambda=(1/2,1/2)$ is
	$\mathrm{mMTP}_2$, Theorem~\ref{thm:markov-product} gives
	$\mathrm{cmMTP}_2(X_3\mid X_1)$. For example,
	with $E=\{0\}$, $F=\{1\}$, $x_3=0$, and $y_3=1$,
	$
	\mathbb P(X_1\in E,X_3\leq0)
	\mathbb P(X_1\in F,X_3\leq1)
	=(0.40)(0.50)=0.20,
	$
	whereas
	$
	\mathbb P(X_1\in E\wedge F,X_3\leq0)
	\mathbb P(X_1\in E\vee F,X_3\leq1)
	=(0.40)(0.50)=0.20.
	$
	Therefore, the defining $\mathrm{cmMTP}_2$ inequality is attained with equality for this choice.
\end{example}

\subsection{Copulas and distributional representations}
\label{subsec:copulas-representations}

We finally express $\mathrm{cmMTP}_2$ at the copula level. Let $F_i$ be the distribution function of $X_i$, $i\in[d]$, and define $U_i:=F_i(X_i)$. If the marginals are continuous, then $\UU=(U_1,\ldots,U_d)$ has uniform marginals and distribution
$C_{\XX}$ by Sklar's theorem. Since the probability integral transform is coordinatewise non-decreasing, the closure result of Theorem~\ref{thm:monotone-transformations} yields a direct measure theoretic connection between $\mathrm{cmMTP}_2(\XX)$ and the copula of $\XX$. When the marginals are strictly increasing, the quantile
transformation gives the converse implication. Thus
$\mathrm{cmMTP}_2$ can be characterized entirely through the copula measure, without requiring the existence of a copula density.

\begin{proposition}[Copula invariance]
	\label{prop:copula-invariance}
	Let $A$ and $B$ be non-empty sets forming a partition of $[d]$.
	Suppose that $F_i$, $i\in[d]$, are continuous and strictly
	increasing, and set $U_i:=F_i(X_i)$. Then
	$\mathrm{cmMTP}_2(\XX_B\mid\XX_A)$ holds if and only if
	$\mathrm{cmMTP}_2(\UU_B\mid\UU_A)$ holds.
\end{proposition}

\begin{proof}
	Suppose first that
	$\mathrm{cmMTP}_2(\XX_B\mid\XX_A)$ holds. Since each
	$F_i$ is non-decreasing, Theorem~\ref{thm:monotone-transformations}
	yields
	$\mathrm{cmMTP}_2(\UU_B\mid\UU_A)$. Conversely, let $Q_i:=F_i^{-1}$, $i\in[d]$. Since, $F_i$ is
	continuous and strictly increasing, $Q_i$ is strictly increasing and
	$X_i=Q_i(U_i)$ almost surely. Hence
	Theorem~\ref{thm:monotone-transformations}, applied to $\UU$ and $(Q_i)_{i\in[d]}$, yields
	$\mathrm{cmMTP}_2(\XX_B\mid\XX_A)$.
\end{proof}

\noindent Proposition~\ref{prop:copula-invariance} shows that, under continuous
strictly increasing margins, $\mathrm{cmMTP}_2$ depends only on the copula measure. We next specialize this representation to Archimedean copulas, whose multivariate construction is characterized
through $d$-monotone generators \citep{McNeilNeslehova2009,DuranteSempi2015}. Let $\UU=(U_1,\ldots,U_d)$ have a strict
Archimedean copula $C_\psi$ with generator
$\psi:[0,\infty)\to(0,1]$ and inverse generator
$\varphi:=\psi^{-1}$. Therefore,
\begin{equation}
	C_\psi(\uu)
	=
	\psi\left(\sum_{i=1}^{d}\varphi(u_i)\right),
	\qquad
	\uu\in(0,1]^d.
	\label{eq:archimedean-copula}
\end{equation}
Let $A$ and $B$ form a partition of $[d]$, with
$|A|=k$ and $|B|=d-k$, and write
$s_A(\uu_A):=\sum_{i\in A}\varphi(u_i)$ and
$s_B(\uu_B):=\sum_{j\in B}\varphi(u_j)$.
Whenever $\psi$ is $k$-times differentiable, set
$h_k(t):=(-1)^k\psi^{(k)}(t)$, $t>0$. For Archimedean copulas, the conditional lower-orthant probabilities are determined by $h_k$ evaluated at $s_A(\uu_A)$ and
$s_A(\uu_A)+s_B(\uu_B)$. Log-convexity conditions on derivatives of Archimedean generators also characterize several classical positive-dependence properties \citep{MullerScarsini2005}. The corresponding criterion for $\mathrm{cmMTP}_2$ is given next.

\begin{theorem}
	\label{thm:archimedean-cmmtp2}
	Let $\UU=(U_1,\ldots,U_d)$ have a strict Archimedean copula $C_\psi$, and let $A$ and $B$ form a partition of $[d]$ with
	$|A|=k$. Suppose $\psi$ is $k$-times continuously differentiable on $(0,\infty)$, $h_k:=(-1)^k\psi^{(k)}>0$, and the conditional
	mapping required in Theorem~\ref{thm:density-equivalence} is
	continuous. Then the following statements are equivalent
	\begin{enumerate}
		\item[(i)] $\mathrm{cmMTP}_2(\UU_B\mid\UU_A)$;
		\item[(ii)] $\mathrm{cMTP}_2(\UU_B\mid\UU_A)$;
		\item[(iii)] $h_k$ is log-convex on $(0,\infty)$.
	\end{enumerate}
\end{theorem}

\begin{proof}
	Since $U_i\sim\operatorname{Unif}(0,1)$, $i\in A$, we have
	$\nu_A=\bigotimes_{i\in A}\mathbb P_{U_i}
	=\lambda^k|_{(0,1)^k}$, whose density is strictly positive and
	continuous. Hence, under the stated regularity assumptions,
	Theorem~\ref{thm:density-equivalence} gives
	$
	\mathrm{cmMTP}_2(\UU_B\mid\UU_A)
	\Longleftrightarrow
	\mathrm{cMTP}_2(\UU_B\mid\UU_A).
	$
	It remains to prove $(ii)\Longleftrightarrow(iii)$.
	For $\uu_A\in(0,1)^k$ and $\uu_B\in(0,1]^{d-k}$,
	differentiating \eqref{eq:archimedean-copula} with respect to $\uu_A$ yields
	\begin{equation}
		\mathbb P(\UU_B\leq\uu_B\mid\UU_A=\uu_A)c_A(\uu_A)
		=
		h_k\!\left(s_A(\uu_A)+s_B(\uu_B)\right)
		\prod_{i\in A}\{-\varphi'(u_i)\},
		\label{eq:archimedean-conditional-map}
	\end{equation}
	where
	$c_A(\uu_A)
	=h_k(s_A(\uu_A))
	\prod_{i\in A}\{-\varphi'(u_i)\}$.
	Therefore, the mapping defining $\mathrm{cMTP}_2$ is
	$$
	g(\uu_A,\uu_B)
	=
	h_k\!\left(s_A(\uu_A)+s_B(\uu_B)\right)
	\prod_{i\in A}\{-\varphi'(u_i)\}.
	$$
	The product factor is lattice-neutral, since, for
	$\uu_A,\vv_A\in(0,1)^k$,
	\[
	\prod_{i\in A}\{-\varphi'(u_i)\}
	\prod_{i\in A}\{-\varphi'(v_i)\}
	=
	\prod_{i\in A}\{-\varphi'(u_i\wedge v_i)\}
	\prod_{i\in A}\{-\varphi'(u_i\vee v_i)\}.
	\]
Consequently, the $\mathrm{MTP}_2$ inequality for $g$ reduces to
\begin{equation}
	\begin{split}
		h_k\!\left(s_A(\uu_A)+s_B(\uu_B)\right)
		h_k\!\left(s_A(\vv_A)+s_B(\vv_B)\right)&\leq
		h_k\!\left(s_A(\uu_A\wedge\vv_A)
		+s_B(\uu_B\wedge\vv_B)\right)\\
		&\hspace{1cm}\times
		h_k\!\left(s_A(\uu_A\vee\vv_A)
		+s_B(\uu_B\vee\vv_B)\right).
	\end{split}
	\label{eq:archimedean-hk-lattice}
\end{equation}
Since $\varphi$ is decreasing,
$s_A(\uu_A\vee\vv_A)+s_B(\uu_B\vee\vv_B)
\leq
\min\{r,s\}\leq
\max\{r,s\}\leq
s_A(\uu_A\wedge\vv_A)+s_B(\uu_B\wedge\vv_B)$,
where
$r:=s_A(\uu_A)+s_B(\uu_B)$ and
$s:=s_A(\vv_A)+s_B(\vv_B)$, while the sums of the two extreme arguments and of $r,s$ coincide. Hence,
\eqref{eq:archimedean-hk-lattice} is equivalent to convexity of
$\log h_k$ \citep{MullerScarsini2005,FuchsWang2027}. Thus
$(ii)\Longleftrightarrow(iii)$, and the assertion follows. Combining this with
	$(i)\Longleftrightarrow(ii)$ proves the assertion.
\end{proof}

\noindent Theorem~\ref{thm:archimedean-cmmtp2} reduces the conditional measure theoretic property to a one-dimensional generator criterion,
which can be checked directly for standard Archimedean families.

\begin{corollary}[Clayton copulas]
	\label{cor:clayton-cmmtp2}
	Let $\UU$ have the $d$-dimensional Clayton copula with parameter
	$\theta>0$. Then, for every non-trivial partition $(A,B)$ of
	$[d]$,
	$\mathrm{cmMTP}_2(\UU_B\mid\UU_A)$ holds.
\end{corollary}

\begin{proof}
	For $\theta>0$, the Clayton generator is
	$\psi_\theta(t)=(1+t)^{-1/\theta}$. If $|A|=k$, then
	\begin{equation}
		h_k(t)
		=
		(-1)^k\psi_\theta^{(k)}(t)
		=
		\left\{\prod_{j=0}^{k-1}
		\left(\frac{1}{\theta}+j\right)\right\}
		(1+t)^{-1/\theta-k}.
		\label{eq:clayton-hk}
	\end{equation}
	Hence,
	$(\log h_k)''(t)=(1/\theta+k)(1+t)^{-2}>0$ for $t>0$.
	Therefore, $h_k$ is log-convex, and
	Theorem~\ref{thm:archimedean-cmmtp2} yields
	$\mathrm{cmMTP}_2(\UU_B\mid\UU_A)$.
\end{proof}

\noindent The generator criterion also permits a compact comparison of standard Archimedean families. Table~\ref{tab:archimedean-cmmtp2} shows the
families for which log-convexity of $h_k=(-1)^k\psi^{(k)}$ is available throughout the indicated parameter range.

\begin{table}[ht]
	\centering
	\caption{Conditional measure-theoretic total positivity for standard
		Archimedean copulas.}
	\label{tab:archimedean-cmmtp2}
	\begin{tabular}{lll}
		\hline
		Family & Parameter & $\mathrm{cmMTP}_2(\UU_B\mid\UU_A)$\\
		\hline
		Clayton & $\theta>0$ & holds\\
		Gumbel-Hougaard & $\theta\geq1$ & holds\\
		Frank & $\theta>0$ & holds\\
		Joe & $\theta\geq1$ & holds\\
		\hline
	\end{tabular}
\end{table}

\noindent For Clayton, the assertion follows from
Corollary~\ref{cor:clayton-cmmtp2}. For the Gumbel-Hougaard and Frank families, $h_k=(-1)^k\psi^{(k)}$ is log-convex for
$k\geq1$ in the stated positive-dependence parameter ranges \citep{FuchsTschimpke2022}. Hence,
Theorem~\ref{thm:archimedean-cmmtp2} yields
$\mathrm{cmMTP}_2(\UU_B\mid\UU_A)$ for every non-trivial
partition $(A,B)$. For the Joe family, $\psi_\theta(t)=1-(1-e^{-t})^{1/\theta}$,
$\theta\geq1$. The generator $\psi_\theta$ is completely monotone on $[0,\infty)$ \citep{Joe2014,FuchsWang2027}. Hence, $h_k=(-1)^k\psi_\theta^{(k)}$ is completely monotone, and therefore, log-convex, for every $k\geq1$. Therefore,
Theorem~\ref{thm:archimedean-cmmtp2} yields
$\mathrm{cmMTP}_2(\UU_B\mid\UU_A)$ for every non-trivial
partition $(A,B)$.

The Gaussian family provides a complementary setting in which conditional measure theoretic total positivity can be related to interpretable matrix restrictions. Let
$\XX\sim N_d(\boldsymbol\mu,\Sigma)$, where $\Sigma\succ0$,
and write $K:=\Sigma^{-1}$. For a non-trivial partition
$A\dot\cup B=[d]$, partition $\Sigma$ and $K$ conformably
and define $\Sigma_{B\mid A}:=\Sigma_{BB}-\Sigma_{BA}\Sigma_{AA}^{-1}\Sigma_{AB}$.
Then $\XX_B\mid\XX_A=\xx_A$ is Gaussian with mean
$\boldsymbol\mu_B+\Sigma_{BA}\Sigma_{AA}^{-1}(\xx_A-\boldsymbol\mu_A)$, and covariance $\Sigma_{B\mid A}=K_{BB}^{-1}$.
The classical Gaussian $\mathrm{dMTP}_2$ criterion is
$K_{ij}\leq0$ for $i\neq j$ \citep{KarlinRinott1980a,FallatEtAl2017}, equivalently, nonnegative partial correlations $\rho_{ij\mid[d]\setminus\{i,j\}}=-K_{ij}/\sqrt{K_{ii}K_{jj}}\geq0$. These identities provide a matrix-based starting point for determining the additional restrictions associated
with $\mathrm{cmMTP}_2(\XX_B\mid\XX_A)$.

\begin{theorem}
	\label{thm:gaussian-cmmtp2}
	Let $\XX\sim N_d(\boldsymbol\mu,\Sigma)$, $\Sigma\succ0$,
	and $A\dot\cup B=[d]$, $A,B\neq\varnothing$.
	Set $S:=\Sigma_{AA}^{-1}$,
	$M:=\Sigma_{BA}\Sigma_{AA}^{-1}$,
	$V:=\Sigma_{BB}-\Sigma_{BA}
	\Sigma_{AA}^{-1}\Sigma_{AB}$, and
	$H(\zz):=\nabla^2\log\Phi_V(\zz)$,
	$\zz\in\Re^B$. Then
	$\mathrm{cmMTP}_2(\XX_B\mid\XX_A)$ holds if and only if,
	for every $\zz\in\Re^B$,
	\begin{equation}
		\begin{split}
			(-S+M^\top H(\zz)M)_{ij}&\geq0,
			\quad i\neq j,\ i,j\in A,\\
			(-M^\top H(\zz))_{ij}&\geq0,
			\quad i\in A,\ j\in B,\\
			H_{ij}(\zz)&\geq0,
			\quad i\neq j,\ i,j\in B.
		\end{split}
		\label{eq:gaussian-conditional-criterion}
	\end{equation}
\end{theorem}

\begin{proof}
	By Gaussian conditioning,
	$\XX_B\mid\XX_A=\xx_A\sim
	N_{|B|}(\boldsymbol\mu_B+
	M(\xx_A-\boldsymbol\mu_A),V)$.
	Thus, with
	$\zz:=\xx_B-\boldsymbol\mu_B-
	M(\xx_A-\boldsymbol\mu_A)$,
	the conditional map in Theorem~
	\ref{thm:density-equivalence} is
	$g(\xx_A,\xx_B)=f_A(\xx_A)\Phi_V(\zz)$.
	Since $g>0$ and $g\in C^2(\Re^d)$,
	$\mathrm{cmMTP}_2(\XX_B\mid\XX_A)$
	is equivalent to
	$\partial_{ij}^2\log g\geq0$ for all
	$i\neq j$, by Theorem~
	\ref{thm:density-equivalence}
	and the differential criterion for
	$\mathrm{MTP}_2$. Direct differentiation gives
	\begin{equation}
		\nabla^2\log g=
		\begin{pmatrix}
			-S+M^\top H(\zz)M&-M^\top H(\zz)\\
			-H(\zz)M&H(\zz)
		\end{pmatrix}.
		\label{eq:gaussian-conditional-hessian}
	\end{equation}
	Nonnegativity of its off-diagonal
	entries is precisely
	\eqref{eq:gaussian-conditional-criterion},
	proving both directions.
\end{proof}

\noindent The Gaussian criterion simplifies when $B=\{b\}$. In this case, $H$ is scalar, and conditional lattice positivity admits an explicit characterization through marginal precision and regression coefficients.

\begin{corollary}
	\label{cor:gaussian-singleton}
	Let $\XX\sim N_d(\boldsymbol\mu,\Sigma)$, $\Sigma\succ0$,
	and $A=[d]\setminus\{b\}$, $B=\{b\}$. Set
	$S:=\Sigma_{AA}^{-1}$,
	$\boldsymbol\beta:=\Sigma_{bA}\Sigma_{AA}^{-1}$,
	and $v:=\Sigma_{bb}-\Sigma_{bA}
	\Sigma_{AA}^{-1}\Sigma_{Ab}>0$.
	Then $\mathrm{cmMTP}_2(X_b\mid\XX_A)$ holds if and
	only if
	\begin{equation}
		S_{ij}+\frac{\beta_i\beta_j}{v}\leq0,
		\ i\neq j,\ i,j\in A,
		\qquad
		\beta_i\geq0,\ i\in A.
		\label{eq:gaussian-singleton-criterion}
	\end{equation}
Equivalently, $K:=\Sigma^{-1}$ is an $M$-matrix, i.e.,
$K_{ij}\leq0$ for all $i\neq j$. More precisely,
$\mathrm{cmMTP}_2(X_b\mid\XX_A)$ is equivalent to
$\mathrm{dMTP}_2(\XX)$ for nondegenerate Gaussian laws.
\end{corollary}

\begin{proof}
	Let $t:=(x_b-\mu_b-
	\boldsymbol\beta(\xx_A-\boldsymbol\mu_A))/\sqrt v$
	and $\ell(t):=\log\Phi(t)$. Since
	$\ell''(t)<0$, Theorem~\ref{thm:gaussian-cmmtp2}
	gives
	$\partial_{ib}^2\log g
	=-\beta_i\ell''(t)/v\geq0$
	if and only if $\beta_i\geq0$.
	Moreover, for $i\neq j$, $i,j\in A$,
	$\partial_{ij}^2\log g
	=-S_{ij}+\beta_i\beta_j\ell''(t)/v$.
	As $t\to-\infty$, $\ell''(t)\to-1$, whereas
	$\ell''(t)\to0$ as $t\to\infty$.
	Consequently, under $\boldsymbol\beta\geq\mathbf0$,
	the required condition is
	$S_{ij}+\beta_i\beta_j/v\leq0$.
	By the block inverse identity,
	$K_{AA}=S+\boldsymbol\beta^\top
	\boldsymbol\beta/v$ and
	$K_{Ab}=-\boldsymbol\beta^\top/v$.
	Hence the conditions are equivalent to
	$K_{ij}\leq0$ for all $i\neq j$.
\end{proof}

\noindent The Gaussian characterization has a direct interpretation in terms of partial correlations. In particular, the singleton
criterion identifies the entire admissible Gaussian class,
rather than merely providing a sufficient condition.

\begin{corollary}
	\label{cor:gaussian-partial-correlation}
	Let $\XX\sim N_d(\boldsymbol\mu,\Sigma)$, $\Sigma\succ0$,
	and $K:=\Sigma^{-1}$. Then the following are equivalent:
	\begin{enumerate}
		\item[(i)] $\mathrm{cmMTP}_2(\XX)$;
		\item[(ii)] $\mathrm{cmMTP}_2(X_b\mid\XX_{[d]\setminus\{b\}})$
		for every $b\in[d]$;
		\item[(iii)] $K_{ij}\leq0$ for all $i\neq j$;
		\item[(iv)] $\rho_{ij\mid[d]\setminus\{i,j\}}\geq0$
		for all $i\neq j$.
	\end{enumerate}
	Consequently, within the nondegenerate Gaussian family,
	$\mathrm{cmMTP}_2(\XX)$ and $\mathrm{dMTP}_2(\XX)$
	are equivalent.
\end{corollary}

\begin{proof}
	By Definition~\ref{def:cmmtp2-vector}, (i) implies (ii). For any
	$b\in[d]$, Corollary~\ref{cor:gaussian-singleton}
	gives
	$\mathrm{cmMTP}_2(X_b\mid\XX_{-b})$
	if and only if $K_{ij}\leq0$, $i\neq j$.
	Thus (ii) implies (iii). Conversely, (iii) implies
	$\mathrm{dMTP}_2(\XX)$ by the classical Gaussian
	criterion \citep{KarlinRinott1980a,FallatEtAl2017}.
	Hence $\mathrm{mMTP}_2(\XX)$, and
	Corollary~2.2 yields (i). Finally,
	$\rho_{ij\mid[d]\setminus\{i,j\}}
	=-K_{ij}/\sqrt{K_{ii}K_{jj}}$ and $K_{ii}>0$
	give (iii) if and only if (iv).
\end{proof}

\noindent For $\Sigma=D_\sigma R D_\sigma$, where
$D_\sigma:=\operatorname{diag}
(\sqrt{\Sigma_{11}},\ldots,\sqrt{\Sigma_{dd}})$,
we have $K=D_\sigma^{-1}R^{-1}D_\sigma^{-1}$.
Thus the Gaussian criterion is equivalently
$(R^{-1})_{ij}\leq0$, $i\neq j$, independently
of marginal scales. In particular,
$\mathrm{cmMTP}_2(\XX)$ imposes nonnegative
pairwise partial correlations, not merely
nonnegative marginal correlations. This gives
a finite matrix-based verification of
Definition~\ref{def:conditional-measure-mtp2} for Gaussian models.

\begin{example}
	\label{ex:gaussian-partial-correlation}
	Let $\XX\sim N_3(\mathbf0,R)$, where
	\begin{equation}
		R=
		\begin{pmatrix}
			1&0.6&0.6\\
			0.6&1&0.1\\
			0.6&0.1&1
		\end{pmatrix},
		\qquad
		R^{-1}=\frac{1}{0.342}
		\begin{pmatrix}
			0.99&-0.54&-0.54\\
			-0.54&0.64&0.26\\
			-0.54&0.26&0.64
		\end{pmatrix}.
		\label{eq:gaussian-correlation-example}
	\end{equation}
	Since $\det(R)=0.342>0$ and all principal minors
	are positive, $R\succ0$. Although $R_{ij}>0$
	for every $i\neq j$,
	$(R^{-1})_{23}=0.26/0.342>0$, yielding
	$\rho_{23\mid1}
	=(0.1-0.6^2)/(1-0.6^2)
	=-0.40625<0$.
	Thus Corollary~\ref{cor:gaussian-partial-correlation}
	implies $\mathrm{cmMTP}_2(\XX)$ fails.
	Moreover, for $A=\{1,2\}$ and $B=\{3\}$,
	Corollary~\ref{cor:gaussian-singleton} gives
	$\mathrm{cmMTP}_2(X_3\mid X_1,X_2)$ false.
	Hence positive pairwise correlations do not
	suffice for conditional lattice positivity;
	the relevant Gaussian restriction is
	$R^{-1}_{ij}\leq0$, $i\neq j$.
\end{example}

\noindent\textit{Gaussian graphical interpretation.}
For $\XX\sim N_d(\boldsymbol\mu,\Sigma)$, $\Sigma\succ0$,
let $K=\Sigma^{-1}$ and
$G=(V,E)$, where $V=[d]$ and
$E=\{\{i,j\}:K_{ij}\neq0,\ i\neq j\}$.
Then $K_{ij}=0$ is equivalent to
$X_i\perp X_j\mid\XX_{[d]\setminus\{i,j\}}$,
whereas $K_{ij}<0$ corresponds to
$\rho_{ij\mid[d]\setminus\{i,j\}}>0$.
By Corollary~\ref{cor:gaussian-partial-correlation},
$\mathrm{cmMTP}_2(\XX)$ holds precisely when
$K_{ij}\leq0$, $i\neq j$. Thus the admissible
Gaussian graphical models have nonnegative
partial correlations along every edge, with
$K_{ij}=0$ encoding conditional independence.
Consequently, verification of
$\mathrm{cmMTP}_2(\XX)$ reduces to
$d(d-1)/2$ sign restrictions on $K$, rather
than the set-indexed inequalities of
Definition~\ref{def:conditional-measure-mtp2}. These restrictions provide
an interpretable covariance-selection
criterion for Gaussian graphical modeling
\citep{KarlinRinott1980a,FallatEtAl2017}.

\begin{remark}
	\label{rem:gaussian-galambos}
	\begin{enumerate}
		\noindent\item[(i)]
		Let $\UU$ have a $d$-dimensional Gaussian copula with correlation matrix $R$. If $R^{-1}$ is an $M$-matrix, i.e.,
		$(R^{-1})_{ij}\leq0$, $i\neq j$, then
		$\mathrm{dMTP}_2(\UU)$ holds
		\citep{KarlinRinott1980a,FallatEtAl2017}. Hence,
		$\mathrm{dMTP}_2(\UU)\Rightarrow
		\mathrm{cMTP}_2(\UU_B\mid\UU_A)
		\Leftrightarrow
		\mathrm{cmMTP}_2(\UU_B\mid\UU_A)$
		for every non-trivial $(A,B)$, where the equivalence is under the hypotheses of Theorem~\ref{thm:density-equivalence}.
		
		\item[(ii)]
		The Galambos family is extreme-value, not Archimedean. Therefore,
		Theorem~\ref{thm:archimedean-cmmtp2} and the criterion $h_k=(-1)^k\psi^{(k)}$ log-convex in
		Table~\ref{tab:archimedean-cmmtp2} do not apply.
		
		\item[(iii)]
		For extreme-value copulas, total positivity is instead formulated through the exponent or stable-tail dependence structure
		\citep{RottgerEtAl2023}. Hence Gaussian and Galambos copulas give two non-Archimedean directions beyond
		Theorem~\ref{thm:archimedean-cmmtp2}.
	\end{enumerate}
\end{remark}

\noindent The preceding results provide both a generator characterization and several parametric classes satisfying $\mathrm{cmMTP}_2$. To make
the Archimedean criterion explicit, consider a strict copula $C_\psi$ and a partition $(A,B)$ with $|A|=k$. By
\eqref{eq:archimedean-conditional-map}, the relevant conditional mapping is
\begin{equation}
	g(\uu_A,\uu_B)
	:=
	\mathbb P(\UU_B\leq\uu_B\mid\UU_A=\uu_A)c_A(\uu_A)
	=
	h_k\!\left(s_A(\uu_A)+s_B(\uu_B)\right)
	\prod_{i\in A}\{-\varphi'(u_i)\},
	\label{eq:archimedean-example-map}
\end{equation}
where $h_k=(-1)^k\psi^{(k)}$. Hence, for
$\uu=(\uu_A,\uu_B)$ and $\vv=(\vv_A,\vv_B)$, the pointwise criterion corresponding to Theorem~\ref{thm:archimedean-cmmtp2} is
\begin{equation}
	g(\uu)g(\vv)
	\leq
	g(\uu\wedge\vv)g(\uu\vee\vv).
	\label{eq:archimedean-example-lattice}
\end{equation}
The next example evaluates \eqref{eq:archimedean-example-lattice}
for a trivariate Clayton copula at two noncomparable points, so that all four lattice terms can be compared numerically.

\begin{example}[Trivariate Clayton copula]
	\label{ex:clayton-numerical}
	Let $\UU=(U_1,U_2,U_3)$ have the trivariate Clayton copula with parameter $\theta=2$, and take $A=\{1,2\}$ and $B=\{3\}$. Then
	$\psi(t)=(1+t)^{-1/2}$ and
	$\varphi(u)=u^{-2}-1$. By Corollary~\ref{cor:clayton-cmmtp2},
	$\mathrm{cmMTP}_2(U_3\mid(U_1,U_2))$ holds. We illustrate the
	corresponding lattice inequality in two steps.
	
	\textit{(i)}
	Since $|A|=2$,
	$h_2(t)=\psi''(t)=\frac{3}{4}(1+t)^{-5/2}$ and
	$-\varphi'(u)=2u^{-3}$. Hence, by
	\eqref{eq:archimedean-conditional-map},
	\begin{equation}
		\begin{split}
			g(u_1,u_2,u_3)
			&:=
			\mathbb P(U_3\leq u_3\mid U_1=u_1,U_2=u_2)c_A(u_1,u_2)=
			3u_1^{-3}u_2^{-3}
			\left(u_1^{-2}+u_2^{-2}+u_3^{-2}-2\right)^{-5/2}.
		\end{split}
		\label{eq:clayton-numerical-g}
	\end{equation}
	Moreover,
	$(\log h_2)''(t)=\frac{5}{2}(1+t)^{-2}>0$; thus
	Theorem~\ref{thm:archimedean-cmmtp2} yields
	$\mathrm{cmMTP}_2(U_3\mid(U_1,U_2))$.
	
	\textit{(ii)}
	Choose $\uu=(0.3,0.7,0.4)$ and
	$\vv=(0.7,0.3,0.6)$. Then $\uu\wedge\vv=(0.3,0.3,0.4)$, and
	$\uu\vee\vv=(0.7,0.7,0.6)$. From
	\eqref{eq:clayton-numerical-g},
	\[
	\begin{array}{c|cccc}
		\ww
		& \uu & \vv & \uu\wedge\vv & \uu\vee\vv\\ \hline
		g(\ww)
		&0.2564&0.4473&1.1413&0.4899
	\end{array}
	\]
	and therefore,
	\begin{equation}
		g(\uu)g(\vv)
		\approx(0.2564)(0.4473)
		=0.1147
		<
		0.5591
		=(1.1413)(0.4899)
		\approx
		g(\uu\wedge\vv)g(\uu\vee\vv).
		\label{eq:clayton-numerical-lattice}
	\end{equation}
	Thus the $\mathrm{MTP}_2$ inequality for the conditional map is strict at this pair of noncomparable points, consistently with
	Theorem~\ref{thm:archimedean-cmmtp2}.
\end{example}

\noindent We now move from particular copula families to a closure result for the copula representation itself. Let $C_n,C$ be $d$-copulas,
$\gamma_n,\gamma$ their induced probability measures on
$[0,1]^d$, and $\UU^{(n)}\sim\gamma_n$, $\UU\sim\gamma$. Since, copulas are continuous distribution functions,
$C_n(\uu)\to C(\uu)$, $\uu\in[0,1]^d$, is equivalent to
$\gamma_n\Rightarrow\gamma$. Hence,
$\gamma_n\Rightarrow\gamma$ and
$\mathrm{cmMTP}_2(\UU_B^{(n)}\mid\UU_A^{(n)})$, $n\geq1$,
suggest the limit property
$\mathrm{cmMTP}_2(\UU_B\mid\UU_A)$. This passage is entirely measure theoretic and requires neither $\gamma_n$ nor $\gamma$ to
admit a density. Moreover, for continuous strictly increasing $F_i$ and $Q_i:=F_i^{-1}$,
$Q:=(Q_i)_{i\in[d]}$ satisfies
$Q(\uu\wedge\vv)=Q(\uu)\wedge Q(\vv)$ and
$Q(\uu\vee\vv)=Q(\uu)\vee Q(\vv)$ on $(0,1)^d$, so the same closedness can be transported through
$\UU^{(n)}\Rightarrow\UU$ to
$Q(\UU^{(n)})\Rightarrow Q(\UU)$. The next Theorem~\ref{thm:closed-copula-representation} combines Theorem~\ref{thm:cmmtp2-weak-closure},
Theorem~\ref{thm:monotone-transformations}, and
Proposition~\ref{prop:copula-invariance} to show that
$\mathrm{cmMTP}_2$ defines a weakly closed copula class whose conditional total positivity is preserved under continuous coordinatewise marginal reconstruction.

\begin{theorem}
	\label{thm:closed-copula-representation}
	Let $A\dot\cup B=[d]$, $A,B\neq\varnothing$, and
	$\UU^{(n)}\sim C_n$, $\UU\sim C$, where $C_n,C$ are
	$d$-copulas. If $C_n\to C$ on $[0,1]^d$ and
	$\mathrm{cmMTP}_2(\UU_B^{(n)}\mid\UU_A^{(n)})$, $n\geq1$, then
	$\UU^{(n)}\Rightarrow\UU$ and
	$\mathrm{cmMTP}_2(\UU_B\mid\UU_A)$. Moreover, for continuous
	strictly increasing distribution functions $F_i$, $i\in[d]$, with
	$Q_i:=F_i^{-1}$, $\XX^{(n)}:=Q(\UU^{(n)})$ and
	$\XX:=Q(\UU)$ satisfy
	$\XX^{(n)}\Rightarrow\XX$ and
	$\mathrm{cmMTP}_2(X_B\mid X_A)$.
\end{theorem}

\begin{proof}
	Define $\gamma_n:=\mathbb P_{\UU^{(n)}}$ and
	$\gamma:=\mathbb P_{\UU}$. Since $C$ is a copula, it is continuous on $[0,1]^d$. Moreover,
	$\gamma_n([0,\uu])=C_n(\uu)\to C(\uu)=\gamma([0,\uu])$ for every
	$\uu\in[0,1]^d$. Hence, the distribution-function characterization of weak convergence yields
	$\gamma_n\Rightarrow\gamma$ \citep{Billingsley1999}. Fix $\xx_B,\yy_B\in[0,1]^B$ and admissible
	$E,F\in\mathcal B([0,1]^A)$. For $n\geq1$, insert
	$\alpha_n(E):=\gamma_n(E\times L_{\xx_B})$,
	$\beta_n(F):=\gamma_n(F\times L_{\yy_B})$,
	$\alpha_n^\wedge(E\wedge F):=
	\gamma_n((E\wedge F)\times L_{\xx_B\wedge\yy_B})$, and
	$\alpha_n^\vee(E\vee F):=
	\gamma_n((E\vee F)\times L_{\xx_B\vee\yy_B})$.
	Since
	$\mathrm{cmMTP}_2(\UU_B^{(n)}\mid\UU_A^{(n)})$,
	Definition~\ref{def:conditional-measure-mtp2} gives
	\begin{equation}
		\alpha_n(E)\beta_n(F)
		\leq
		\alpha_n^{\wedge}(E\wedge F)
		\alpha_n^{\vee}(E\vee F).
		\label{eq:copula-approx-lattice}
	\end{equation}
	By $\gamma_n\Rightarrow\gamma$ and
	Theorem~\ref{thm:cmmtp2-weak-closure}, the weak limit inherits the
	conditional measure-theoretic lattice inequality. Therefore, for each admissible $E,F$ and $\xx_B,\yy_B$,
	\begin{equation}
		\begin{split}
			&\gamma(E\times L_{\xx_B})
			\gamma(F\times L_{\yy_B})\leq
			\gamma((E\wedge F)\times L_{\xx_B\wedge\yy_B})
			\gamma((E\vee F)\times L_{\xx_B\vee\yy_B}).
		\end{split}
		\label{eq:copula-limit-lattice}
	\end{equation}
	Hence, $\mathrm{cmMTP}_2(\UU_B\mid\UU_A)$ by Definition~\ref{def:conditional-measure-mtp2}. For the second assertion, define $Q(\uu):=(Q_i(u_i))_{i\in[d]}$. Since the marginals of every copula are uniform,
	$\mathbb P(\UU^{(n)}\in(0,1)^d)
	=\mathbb P(\UU\in(0,1)^d)=1$. Each $Q_i$ is continuous and strictly increasing on $(0,1)$; hence the continuous mapping theorem
	gives $\XX^{(n)}=Q(\UU^{(n)})\Rightarrow Q(\UU)=\XX$.
	Furthermore, for $\uu,\vv\in(0,1)^d$,
	\begin{equation}
		Q(\uu\wedge\vv)=Q(\uu)\wedge Q(\vv),
		\qquad
		Q(\uu\vee\vv)=Q(\uu)\vee Q(\vv).
		\label{eq:quantile-lattice}
	\end{equation}
	Therefore, $\mathrm{cmMTP}_2(\UU_B\mid\UU_A)$ and
	Theorem~\ref{thm:monotone-transformations} imply
	$\mathrm{cmMTP}_2(X_B\mid X_A)$. This completes the proof.
\end{proof}

\noindent Theorem~\ref{thm:closed-copula-representation} is constructed for a fixed partition $(A,B)$. Applying it simultaneously to $A\dot\cup B=[d]$ yields the corresponding closure property for the whole-vector notion in Definition~\ref{def:cmmtp2-vector}. In particular, weak limits of $\mathrm{cmMTP}_2$ copulas remain $\mathrm{cmMTP}_2$, and this
property is retained after reconstruction of arbitrary continuous strictly increasing margins.

\begin{corollary}
	\label{cor:copula-whole-vector-closure}
	Let $\UU^{(n)}\sim C_n$, $\UU\sim C$, where $C_n,C$ are	$d$-copulas and $C_n\to C$ on $[0,1]^d$. If
	$\mathrm{cmMTP}_2(\UU^{(n)})$ holds for every $n\geq1$, then	$\UU^{(n)}\Rightarrow\UU$ and $\mathrm{cmMTP}_2(\UU)$. Moreover,
	if $F_i$, $i\in[d]$, are continuous and strictly increasing, $Q_i:=F_i^{-1}$, $\XX^{(n)}:=Q(\UU^{(n)})$, and
	$\XX:=Q(\UU)$, then $\XX^{(n)}\Rightarrow\XX$ and $\mathrm{cmMTP}_2(\XX)$.
\end{corollary}

\begin{proof}
	For every $\varnothing\neq A\subsetneq[d]$, define $B:=[d]\setminus A$. Definition~\ref{def:cmmtp2-vector} yields
	$\mathrm{cmMTP}_2(\UU_B^{(n)}\mid\UU_A^{(n)})$, $n\geq1$. Theorem~\ref{thm:closed-copula-representation} therefore yields $\mathrm{cmMTP}_2(\UU_B\mid\UU_A)$ and $\mathrm{cmMTP}_2(X_B\mid X_A)$ for every such $(A,B)$, together with $\UU^{(n)}\Rightarrow\UU$ and
	$\XX^{(n)}\Rightarrow\XX$. Definition~\ref{def:cmmtp2-vector} now implies
	$\mathrm{cmMTP}_2(\UU)$ and $\mathrm{cmMTP}_2(\XX)$.
\end{proof}

\noindent Corollary~\ref{cor:copula-whole-vector-closure} is important when a sequence of regular copula models approaches a singular dependence structure. Indeed, let $C_n\to C$ with associated measures $\gamma_n\Rightarrow\gamma$. Even if $\gamma_n\ll\lambda^d$ for every $n$, it may happen that $\gamma\perp\lambda^d$, so that no limiting copula density is available. Theorem~\ref{thm:closed-copula-representation} and
Corollary~\ref{cor:copula-whole-vector-closure} show that this loss of absolute continuity does not affect $\mathrm{cmMTP}_2$. The following Gaussian-copula approximation makes this phenomenon explicit.

\begin{example}[Singular copula limit]
	\label{ex:singular-copula-limit}
	Let $\rho_n:=1-n^{-1}$, $n\geq2$, and let
	$\UU^{(n)}=(U_1^{(n)},U_2^{(n)})$ have the Gaussian copula
	$C_{\rho_n}$ with correlation matrix
	\[
	R_n=
	\begin{pmatrix}
		1&\rho_n\\
		\rho_n&1
	\end{pmatrix}.
	\]
	Let $M(u,v):=u\wedge v$ be the comonotonic copula. Then $C_{\rho_n}\to M$ pointwise on $[0,1]^2$ and
	$\mathrm{cmMTP}_2(\UU^{(n)})$ holds for every $n\geq2$. We verify the limiting conclusion in two steps.
	
	\textit{(i)}
	For $\rho_n\in(0,1)$,
	\[
	R_n^{-1}
	=
	\frac{1}{1-\rho_n^2}
	\begin{pmatrix}
		1&-\rho_n\\
		-\rho_n&1
	\end{pmatrix},
	\]
	which is an $M$-matrix. Hence the corresponding Gaussian law is $\mathrm{dMTP}_2$ and, by Remark~\ref{rem:gaussian-galambos},
	$\mathrm{cmMTP}_2(\UU^{(n)})$ holds. Moreover,
	$C_{\rho_n}$ is absolutely continuous on $(0,1)^2$ for every $n\geq2$.
	
	\textit{(ii)}
	Let $(Z_1^{(n)},Z_2^{(n)})\sim N(\mathbf 0,R_n)$ and define $U_i^{(n)}:=\Phi(Z_i^{(n)})$. Since
	$\rho_n\uparrow1$, $(Z_1^{(n)},Z_2^{(n)})\Rightarrow(Z,Z)$, $Z\sim N(0,1)$, and
	therefore,
	\[
	\UU^{(n)}
	\Rightarrow
	(\Phi(Z),\Phi(Z))
	=:(U,U).
	\]
	The law $\gamma_M$ of $(U,U)$ is concentrated on
	$\Delta:=\{(u,u):0\leq u\leq1\}$, so that
	$\gamma_M(\Delta)=1$ whereas $\lambda^2(\Delta)=0$. Hence,
	$\gamma_M\perp\lambda^2$. Nevertheless,
	Corollary~\ref{cor:copula-whole-vector-closure} yields $\mathrm{cmMTP}_2(U,U)$. For example, at $(u,v)=(1/2,1/2)$,
	\[
	C_{\rho_n}(1/2,1/2)
	=
	\frac14+\frac{1}{2\pi}\arcsin(\rho_n)
	\longrightarrow
	\frac12
	=
	M(1/2,1/2).
	\]
	Therefore, a sequence of absolutely continuous
	$\mathrm{cmMTP}_2$ copulas converges to the singular comonotonic copula while the measure theoretic conditional total-positivity property is retained.
\end{example}

\section{Inference for conditional lattice inequalities}
\label{sec:inference}

The measure theoretic structure of $\mathrm{cmMTP}_2$ yields sample inference without density estimation. Let
$\XX^{(1)},\ldots,\XX^{(n)}\stackrel{\mathrm{iid}}{\sim}\mu$, $A\dot\cup B=[d]$, $A,B\neq\varnothing$, and fix
$\mathcal Q=\{(C_r,D_r,\xx_{B,r},\yy_{B,r}):r\in[m]\}$ with admissible $C_r,D_r\in\mathcal B_A$ and
$\xx_{B,r},\yy_{B,r}\in\Re^B$. The four probabilities in
Definition~\ref{def:conditional-measure-mtp2} admit empirical-frequency estimators, yielding $\widehat{\boldsymbol\Delta}_n=(\widehat\Delta_{n,r})_{r\in[m]}$ for the corresponding lattice contrasts $\boldsymbol\Delta=(\Delta_r)_{r\in[m]}$. A multivariate
central limit theorem and the delta method determine the joint asymptotic law of
$\sqrt n(\widehat{\boldsymbol\Delta}_n-\boldsymbol\Delta)$, providing density-free multivariate inference for the inequalities $\Delta_r\geq0$, $r\in[m]$, implied by
$\mathrm{cmMTP}_2(\XX_B\mid\XX_A)$. No absolute continuity is required; $\mu$ may be continuous, discrete, mixed, or singular.

\subsection{Empirical lattice contrasts}
\label{subsec:empirical-lattice-contrasts}

For $r\in[m]$, associate the comparison $(C_r,D_r,\xx_{B,r},\yy_{B,r})$ with the four events
$E_{r1}:=C_r\times L_{\xx_{B,r}}$, $E_{r2}:=D_r\times L_{\yy_{B,r}}$, $E_{r3}:=(C_r\wedge D_r)\times L_{\xx_{B,r}\wedge\yy_{B,r}}$, and $E_{r4}:=(C_r\vee D_r)\times L_{\xx_{B,r}\vee\yy_{B,r}}$. Define
$p_{rj}:=\mu(E_{rj})$ and $\pp:=(p_{rj})_{r\in[m],j\in[4]}\in[0,1]^{4m}$. The corresponding population contrast is
$\Delta_r:=p_{r3}p_{r4}-p_{r1}p_{r2}$. Therefore,
$\mathrm{cmMTP}_2(\XX_B\mid\XX_A)$ implies
$\Delta_r\geq0$, $r\in[m]$. From $\XX^{(1)},\ldots,\XX^{(n)}$, estimate $p_{rj}$ by the empirical frequency $\widehat p_{rj}:=n^{-1}\sum_{\ell=1}^n\mathbf 1\{\XX^{(\ell)}\in E_{rj}\}$ and define $\widehat\Delta_{n,r}:=\widehat p_{r3}\widehat p_{r4}-\widehat p_{r1}\widehat p_{r2}$. For the joint asymptotic analysis,
let $\ZZ^{(\ell)}:=(\mathbf 1\{\XX^{(\ell)}\in E_{rj}\})
_{r\in[m],j\in[4]}$, and $\widehat\pp_n:=n^{-1}\sum_{\ell=1}^n\ZZ^{(\ell)}$.
Then $\mathbb E\ZZ^{(1)}=\pp$ and $\widehat{\boldsymbol\Delta}_n=(\widehat\Delta_{n,r})_{r\in[m]}$ is a smooth function of $\widehat\pp_n$.

Let $\Sigma:=\operatorname{Cov}(\ZZ^{(1)})$. Its entries are $\Sigma_{(r,j),(s,k)}=\mu(E_{rj}\cap E_{sk})-p_{rj}p_{sk}$ for $(r,j),(s,k)\in[m]\times[4]$. Define $G:[0,1]^{4m}\to\Re^m$ by
$G_r(\pp):=p_{r3}p_{r4}-p_{r1}p_{r2}$, $r\in[m]$. Thus
$\boldsymbol\Delta=G(\pp)$ and $\widehat{\boldsymbol\Delta}_n=G(\widehat\pp_n)$. Writing
$J(\pp):=DG(\pp)$, the $r$th row of $J(\pp)$ is zero outside the four coordinates indexed by $(r,1),\ldots,(r,4)$ and equals
$(-p_{r2},-p_{r1},p_{r4},p_{r3})$ on those coordinates. Define $\Omega:=J(\pp)\Sigma J(\pp)^\top$. Theorem~\ref{thm:lattice-contrast-clt} derives the joint asymptotic distribution of the empirical lattice contrasts and provides the Gaussian basis for the
simultaneous inference developed in Subsection~\ref{subsec:simultaneous-testing}.

\begin{theorem}
	\label{thm:lattice-contrast-clt}
	For fixed $m$ and $\mathcal Q$ as above,
	\begin{equation}
		\sqrt n\bigl(
		\widehat{\boldsymbol\Delta}_n-\boldsymbol\Delta
		\bigr)
		\Rightarrow N_m(\mathbf 0,\Omega).
		\label{eq:lattice-contrast-clt}
	\end{equation}
	No absolute-continuity or continuity assumptions on $\mu$ are
	required.
\end{theorem}

\begin{proof}
	Since $\ZZ^{(\ell)}\in\{0,1\}^{4m}$ are i.i.d. with
	$\mathbb E\ZZ^{(\ell)}=\pp$ and
	$\operatorname{Cov}(\ZZ^{(\ell)})=\Sigma$, the multivariate central limit theorem yields
	$\sqrt n(\widehat\pp_n-\pp)\Rightarrow
	N_{4m}(\mathbf0,\Sigma)$ and hence,
	$\widehat\pp_n-\pp=O_p(n^{-1/2})$.
	Put $\hh_n:=\widehat\pp_n-\pp$. For $r\in[m]$, direct expansion of
	$G_r$ gives
	\begin{equation}
		G_r(\pp+\hh_n)-G_r(\pp)
		=
		\nabla G_r(\pp)^\top\hh_n
		+h_{n,r3}h_{n,r4}-h_{n,r1}h_{n,r2}.
		\label{eq:delta-lattice-expansion}
	\end{equation}
	Since $\hh_n=O_p(n^{-1/2})$, the last two terms in
	\eqref{eq:delta-lattice-expansion} are $O_p(n^{-1})$.
	As $m<\infty$,
	$\sqrt n\{G(\widehat\pp_n)-G(\pp)\}
	=J(\pp)\sqrt n(\widehat\pp_n-\pp)+o_p(1)$.
	Slutsky's theorem therefore yields
	$
	\sqrt n\{G(\widehat\pp_n)-G(\pp)\}
	\Rightarrow
	N_m(\mathbf0,J(\pp)\Sigma J(\pp)^\top).
	$
	Since, $G(\widehat\pp_n)=\widehat{\boldsymbol\Delta}_n$,
	$G(\pp)=\boldsymbol\Delta$, and
	$\Omega=J(\pp)\Sigma J(\pp)^\top$, \eqref{eq:lattice-contrast-clt}
	follows.
\end{proof}

\noindent For implementation, the covariance matrix $\Omega$ in
Theorem~\ref{thm:lattice-contrast-clt} must be estimated from the same empirical event probabilities. The following Corollary~\ref{cor:lattice-contrast-studentization} provides a consistent plug-in estimator and the corresponding studentized limits.

\begin{corollary}
	\label{cor:lattice-contrast-studentization}
	Let
	$\widehat\Sigma_n:=n^{-1}\sum_{\ell=1}^n
	(\ZZ^{(\ell)}-\widehat\pp_n)
	(\ZZ^{(\ell)}-\widehat\pp_n)^\top$, and define
	$\widehat\Omega_n:=
	J(\widehat\pp_n)\widehat\Sigma_n
	J(\widehat\pp_n)^\top$. Then
	$\widehat\Omega_n\stackrel{p}{\longrightarrow}\Omega$.
	Moreover, if $\sigma_r^2:=\Omega_{rr}>0$ and
	$\widehat\sigma_{n,r}^2:=(\widehat\Omega_n)_{rr}$, then, for
	$r\in[m]$,
	\begin{equation}
		\frac{\sqrt n(\widehat\Delta_{n,r}-\Delta_r)}
		{\widehat\sigma_{n,r}}
		\Rightarrow N(0,1).
		\label{eq:studentized-lattice-contrast}
	\end{equation}
\end{corollary}

\begin{proof}
	By the law of large numbers,
	$\widehat\pp_n\stackrel{p}{\to}\pp$ and
	$\widehat\Sigma_n\stackrel{p}{\to}\Sigma$.
	Continuity of $J$ gives
	$J(\widehat\pp_n)\stackrel{p}{\to}J(\pp)$. Hence,
	$\widehat\Omega_n\stackrel{p}{\to}
	J(\pp)\Sigma J(\pp)^\top=\Omega$.
	Thus $\widehat\sigma_{n,r}\stackrel{p}{\to}\sigma_r>0$, and
	\eqref{eq:studentized-lattice-contrast} follows from
	Theorem~\ref{thm:lattice-contrast-clt} and Slutsky's theorem.
\end{proof}

\subsection{Simultaneous testing}
\label{subsec:simultaneous-testing}

The finite system $\mathcal Q$ yields the testable restriction $\boldsymbol\Delta\in[0,\infty)^m$ implied by
$\mathrm{cmMTP}_2(\XX_B\mid\XX_A)$. Accordingly, consider
$H_0:\Delta_r\geq0$, $r\in[m]$, against
$H_1:\min_{r\in[m]}\Delta_r<0$. Therefore, $H_1$ corresponds to at least one lattice violation in $\mathcal Q$; for fixed $m$, this is a joint
test of the selected inequalities, not of the full family in Definition~\ref{def:conditional-measure-mtp2}. Let $\sigma_r^2:=\Omega_{rr}>0$ and
$\widehat\sigma_{n,r}^2:=(\widehat\Omega_n)_{rr}$, $r\in[m]$, and define the maximum standardized violation
\begin{equation}
	T_n:=
	\max_{r\in[m]}
	\frac{\sqrt n[-\widehat\Delta_{n,r}]_+}
	{\widehat\sigma_{n,r}},
	\label{eq:max-lattice-test}
\end{equation}
where $[a]_+:=\max\{a,0\}$. Let $D:=\operatorname{diag}(\sigma_1,\ldots,\sigma_m)$,
$R:=D^{-1}\Omega D^{-1}$, and let
$\widehat R_n$ be its plug-in estimator obtained from
$\widehat\Omega_n$. For $Z\sim N_m(\mathbf0,R)$, denote by $c_{1-\alpha}(R)$ the $(1-\alpha)$-quantile of
$\max_{r\in[m]}[-Z_r]_+$. The feasible test rejects $H_0$ when $T_n>c_{1-\alpha}(\widehat R_n)$. Proposition~\ref{prop:simultaneous-test} establishes asymptotic size control at the composite null,
exactness at its least-favorable boundary, and consistency under fixed lattice violations.

\begin{proposition}
	\label{prop:simultaneous-test}
	Suppose $m<\infty$, $\sigma_r^2>0$, $r\in[m]$, and let
	$\varphi_n:=\mathbf 1\{T_n>c_{1-\alpha}(\widehat R_n)\}$.
	Assume that the distribution of
	$\max_{r\in[m]}[-Z_r]_+$, $Z\sim N_m(\mathbf0,R)$, is continuous
	at $c_{1-\alpha}(R)$. Then
	\begin{enumerate}
		\item[(i)] if $\boldsymbol\Delta\in[0,\infty)^m$,
		$\limsup_{n\to\infty}\mathbb P_\mu(\varphi_n=1)\leq\alpha$;
		\item[(ii)] if $\boldsymbol\Delta=\mathbf0$,
		$\mathbb P_\mu(\varphi_n=1)\to\alpha$;
		\item[(iii)] if $\min_{r\in[m]}\Delta_r<0$,
		$\mathbb P_\mu(\varphi_n=1)\to1$.
	\end{enumerate}
\end{proposition}

\begin{proof}
	Let $I_0:=\{r\in[m]:\Delta_r=0\}$. Under $H_0$, for
	$r\notin I_0$, $\sqrt n\,\widehat\Delta_{n,r}/
	\widehat\sigma_{n,r}\stackrel{p}{\to}\infty$, whereas, by Theorem~\ref{thm:lattice-contrast-clt} and Corollary~\ref{cor:lattice-contrast-studentization},
	$(\sqrt n\,\widehat\Delta_{n,r}/\widehat\sigma_{n,r})_{r\in I_0}$ converges to the corresponding subvector of $Z\sim
	N_m(\mathbf0,R)$. Hence the asymptotic rejection probability is bounded by that obtained when $I_0=[m]$, i.e., $\boldsymbol\Delta=\mathbf0$. Since, $\widehat R_n\stackrel{p}{\to}R$,
	$c_{1-\alpha}(\widehat R_n)\stackrel{p}{\to}
	c_{1-\alpha}(R)$, and the assumed continuity at
	$c_{1-\alpha}(R)$ yields (i). If $I_0=[m]$, the bound is attained, which gives (ii). If $\Delta_{r_0}<0$ for some $r_0$, then
	$\sqrt n[-\widehat\Delta_{n,r_0}]_+/
	\widehat\sigma_{n,r_0}\stackrel{p}{\to}\infty$.
	Therefore, $T_n\stackrel{p}{\to}\infty$, whereas
	$c_{1-\alpha}(\widehat R_n)\stackrel{p}{\to}
	c_{1-\alpha}(R)<\infty$, proving (iii).
\end{proof}

\noindent Proposition~\ref{prop:simultaneous-test} concerns fixed nulls and alternatives. The local behavior of the procedure is obtained by allowing the lattice contrasts to approach the boundary at the $n^{-1/2}$ rate.

\begin{theorem}
	\label{thm:local-power}
	Let $(\mu_n)$ be a sequence of laws with lattice contrasts
	$\boldsymbol\Delta_n$ and suppose
	$\sqrt n(\widehat{\boldsymbol\Delta}_n-\boldsymbol\Delta_n)
	\Rightarrow N_m(\mathbf0,\Omega)$,
	$\sqrt n\,\boldsymbol\Delta_n\to\boldsymbol\delta\in\Re^m$, and
	$\widehat\Omega_n\stackrel{p}{\to}\Omega$, where
	$\sigma_r^2:=\Omega_{rr}>0$, $r\in[m]$. Let
	$D:=\operatorname{diag}(\sigma_1,\ldots,\sigma_m)$ and
	$R:=D^{-1}\Omega D^{-1}$. Then
	\begin{equation}
		T_n\Rightarrow
		T(\boldsymbol\delta):=
		\max_{r\in[m]}
		\left[-Z_r-\frac{\delta_r}{\sigma_r}\right]_+,
		\quad
		\ZZ\sim N_m(\mathbf0,R).
		\label{eq:local-power}
	\end{equation}
	If the distribution of $T(\boldsymbol\delta)$ is continuous at
	$c_{1-\alpha}(R)$, then
	$\mathbb P_{\mu_n}(\varphi_n=1)\to
	\mathbb P\{T(\boldsymbol\delta)>c_{1-\alpha}(R)\}$.
\end{theorem}

\begin{proof}
	Let $\widehat D_n:=
	\operatorname{diag}(\widehat\sigma_{n,1},\ldots,
	\widehat\sigma_{n,m})$. Since
	$\widehat\Omega_n\stackrel{p}{\to}\Omega$ and
	$\min_{r\in[m]}\sigma_r^2>0$,
	$\widehat D_n\stackrel{p}{\to}D$. Furthermore,
	\[
	\widehat D_n^{-1}\sqrt n\,\widehat{\boldsymbol\Delta}_n
	=
	\widehat D_n^{-1}
	\sqrt n(\widehat{\boldsymbol\Delta}_n-\boldsymbol\Delta_n)
	+
	\widehat D_n^{-1}\sqrt n\,\boldsymbol\Delta_n
	\Rightarrow
	\ZZ+D^{-1}\boldsymbol\delta,
	\]
	by Slutsky's theorem. Since
	$\zz\mapsto\max_{r\in[m]}[-z_r]_+$ is continuous,
	the continuous mapping theorem yields
	\eqref{eq:local-power}. Moreover,
	$\widehat R_n\stackrel{p}{\to}R$ implies
	$c_{1-\alpha}(\widehat R_n)\stackrel{p}{\to}
	c_{1-\alpha}(R)$. Continuity of the distribution of
	$T(\boldsymbol\delta)$ at $c_{1-\alpha}(R)$ therefore yields
	$\mathbb P_{\mu_n}(\varphi_n=1)\to
	\mathbb P\{T(\boldsymbol\delta)>c_{1-\alpha}(R)\}$.
\end{proof}

\noindent The preceding results make the finite-system test directly implementable from empirical event frequencies. The following numerical illustration shows how simultaneous lattice violations are detected from
$\widehat{\boldsymbol\Delta}_n$ and $\widehat\Omega_n$, without estimating either a joint or conditional density.

\begin{example}[Simultaneous lattice test]
	\label{ex:simultaneous-lattice-test}
	Let $n=500$, $m=2$, and suppose the empirical probabilities associated
	with two comparisons in $\mathcal Q$ yield
	$\widehat{\boldsymbol\Delta}_n=(-0.020,0.008)^\top$ and
	$\widehat\Omega_n=\operatorname{diag}(0.0400,0.0256)$.
	Thus
	$\widehat\sigma_{n,1}=0.20$,
	$\widehat\sigma_{n,2}=0.16$, and $\widehat R_n=I_2$. From
	\eqref{eq:max-lattice-test},
	\[
	T_n
	=
	\max\left\{
	\frac{\sqrt{500}(0.020)}{0.20},
	0
	\right\}
	=2.236.
	\]
	For $\alpha=0.05$ and $\ZZ\sim N_2(\mathbf0,I_2)$,
	$c_{0.95}(I_2)=\Phi^{-1}(\sqrt{0.95})\approx1.955$.
	Hence $T_n>c_{0.95}(I_2)$ and $H_0$ is rejected. The first comparison
	therefore provides evidence of a lattice violation within
	$\mathcal Q$. This conclusion concerns the selected finite system and
	does not, by itself, constitute an omnibus test of
	$\mathrm{cmMTP}_2(\XX_B\mid\XX_A)$.
\end{example}

\noindent\textit{Statistical interpretation.}
The hypotheses in Subsection~\ref{subsec:simultaneous-testing} concern
$H_0:\boldsymbol\Delta\in[0,\infty)^m$, a necessary
finite-dimensional restriction implied by
$\mathrm{cmMTP}_2(\XX_B\mid\XX_A)$, rather than
Definition~\ref{def:conditional-measure-mtp2} in its entirety. Thus
$\min_{r\in[m]}\Delta_r<0$ implies
$\neg\mathrm{cmMTP}_2(\XX_B\mid\XX_A)$, whereas
$\boldsymbol\Delta\in[0,\infty)^m$ does not establish
$\mathrm{cmMTP}_2(\XX_B\mid\XX_A)$.
Consequently, rejection of $H_0$ at asymptotic level
$\alpha$ provides statistical evidence against
$\mathrm{cmMTP}_2$ and excludes, at the corresponding
error level, models whose probability laws satisfy
Definition~\ref{def:conditional-measure-mtp2}. Non-rejection establishes neither
$\mathrm{cmMTP}_2$ nor its structural consequences.
For example, rejection rules out the joint
$\mathrm{cmMTP}_2$ restriction imposed by a
Gaussian model with $K=\Sigma^{-1}$ an $M$-matrix,
or by an Archimedean copula satisfying the
log-convexity criterion of Theorem~
\ref{thm:archimedean-cmmtp2}, provided the
selected comparisons and model assumptions
are appropriate. Hence, the procedure serves as
a density-free specification diagnostic for
positively dependent multivariate models, while
formal verification of Definition~\ref{def:conditional-measure-mtp2} requires
control over an increasingly rich collection of
lattice inequalities.

\noindent\textit{Confidence bounds and model assessment.}
For $r\in[m]$, let
$\widehat\sigma_{n,r}^2=(\widehat\Omega_n)_{rr}>0$
and define the simultaneous lower confidence bounds
$L_{n,r}:=\widehat\Delta_{n,r}
-c_{1-\alpha}(\widehat R_n)
\widehat\sigma_{n,r}/\sqrt n$.
By Theorem~\ref{thm:lattice-contrast-clt} and Corollary~\ref{cor:lattice-contrast-studentization}, under the continuity condition of Proposition~\ref{prop:simultaneous-test},
\begin{equation}
	\mathbb P_\mu\left(
	\Delta_r\geq L_{n,r},\ r\in[m]
	\right)\longrightarrow1-\alpha.
	\label{eq:simultaneous-lattice-bounds}
\end{equation}
Hence, $L_{n,r}>0$, $r\in[m]$, provides simultaneous
evidence for the selected inequalities, but does not imply
$\mathrm{cmMTP}_2(\XX_B\mid\XX_A)$. To detect violations,
define the simultaneous upper bounds
$U_{n,r}:=\widehat\Delta_{n,r}
+c_{1-\alpha}(\widehat R_n)
\widehat\sigma_{n,r}/\sqrt n$.
Under the same regularity conditions,
$\mathbb P_\mu(\Delta_r\leq U_{n,r},\
r\in[m])\to1-\alpha$.
Hence $\min_{r\in[m]}U_{n,r}<0$ rejects the
corresponding finite-system restriction and provides
evidence against Definition~\ref{def:conditional-measure-mtp2}. For
$\mathcal M_+:=\{\mu_\vartheta:\vartheta\in\Theta,\
\mathrm{cmMTP}_2(\mu_\vartheta)\}$, such rejection
excludes $\mathcal M_+$ as a correctly specified
model class at asymptotic level $\alpha$.
In particular, Gaussian models with
$(R^{-1})_{ij}\leq0$, $i\neq j$, and Archimedean
models satisfying the criterion of
Theorem~\ref{thm:archimedean-cmmtp2}
are subject to these restrictions.
Conversely, non-rejection does not establish
$\mu\in\mathcal M_+$, since $m<\infty$.
Thus the empirical inequalities provide
density-free specification diagnostics rather
than a finite-sample characterization of
$\mathrm{cmMTP}_2$.

\section{Discussion}
\label{sec:discussion}

\noindent\textit{Main findings and implications.}
This paper develops a density-free theory of conditional
multivariate total positivity through the set-lattice
formulation $\mathrm{cmMTP}_2(\XX_B\mid\XX_A)$.
The principal distinction from distribution-function
$\mathrm{MTP}_2$ is the additional restriction on
$\mu^A_{\xx_B}(C)=\mathbb P(\XX_A\in C,\XX_B\leq\xx_B)$
for arbitrary admissible conditioning sets $C,D\in\mathcal B_A$,rather than marginal stability alone. The resulting hierarchy$\mathrm{mMTP}_2(\XX)\Longrightarrow
\mathrm{cmMTP}_2(\XX_B\mid\XX_A)\Longrightarrow
\mathrm{MTP}_2(\XX)$
extends conditional total positivity to discrete, mixed,
and singular laws, while Theorem~\ref{thm:density-equivalence}
establishes equivalence with density-based
$\mathrm{cMTP}_2$ under regularity conditions.
The property is preserved under marginalization,
coordinatewise non-decreasing transformations,
and weak convergence, including
$\mu_n\ll\lambda^d$, $\mu_n\Rightarrow\mu$,
and $\mu\perp\lambda^d$.
These stability results extend to lattice-positive
Markov-kernel composition and copula approximation.
For strict Archimedean copulas,
$\mathrm{cmMTP}_2(\UU_B\mid\UU_A)$ is characterized
by log-convexity of
$h_k=(-1)^k\psi^{(k)}$, $k=|A|$.
For $\XX\sim N_d(\boldsymbol\mu,\Sigma)$,
$\Sigma\succ0$, Corollary~\ref{cor:gaussian-partial-correlation}
gives the explicit characterization
$\mathrm{cmMTP}_2(\XX)
\Longleftrightarrow
K_{ij}\leq0$, $i\neq j$, $K=\Sigma^{-1}$,
equivalently,
$\rho_{ij\mid[d]\setminus\{i,j\}}\geq0$.
Thus, within the nondegenerate Gaussian family,
$\mathrm{cmMTP}_2$ coincides with
$\mathrm{dMTP}_2$ and admits verification through
$d(d-1)/2$ precision-matrix restrictions,
providing a direct interpretation for Gaussian
graphical models \citep{KarlinRinott1980a,FallatEtAl2017}.
Finally, for a fixed finite system $\mathcal Q$,
$\sqrt n(\widehat{\boldsymbol\Delta}_n-
\boldsymbol\Delta)\Rightarrow N_m(\mathbf0,\Omega)$
yields density-free simultaneous inference.
Rejection of
$H_0:\boldsymbol\Delta\in[0,\infty)^m$
provides evidence against
$\mathrm{cmMTP}_2$ and excludes incompatible
positive-dependence model restrictions,
whereas non-rejection does not establish
Definition~\ref{def:conditional-measure-mtp2}. Consequently, the framework
connects conditional dependence structure,
tractable distributional characterization,
and multivariate model assessment, extending
the measure-theoretic approach of
\citet{ColangeloEtAl2006} and the conditional
theory of \citet{FuchsWang2027}.

\noindent\textit{Further directions.}
Several questions remain open. First, the Gaussian
characterization suggests studying whether
$\mathrm{cmMTP}_2$ admits comparably explicit
matrix or parameter restrictions for other
elliptical, mixture, and graphical families,
and whether the resulting restrictions distinguish
conditional lattice positivity from
distribution-function $\mathrm{MTP}_2$.
Second, the sufficient composition conditions
$K_1\in\mathcal K_{\mathrm{mLTP}_2}(A,B)$ and
$K_2\in\mathcal K_{\mathrm{LTP}_2}(B,C)$
motivate characterization of weaker or maximal
kernel classes preserving $\mathrm{cmMTP}_2$,
particularly for vector-valued intermediate blocks.
Third, Section~\ref{sec:inference} considers a fixed finite collection
$\mathcal Q$, for which
$\boldsymbol\Delta\geq\mathbf0$ is necessary but
not sufficient for Definition~\ref{def:conditional-measure-mtp2}.
An omnibus procedure with
$\mathcal Q=\mathcal Q_n$, $|\mathcal Q_n|\to\infty$,
would require uniform empirical-process
approximations, suitable bootstrap calibration,
and control of data-driven lattice selection.
Such extensions could provide stronger
specification diagnostics and guide constrained
covariance estimation, graphical model selection,
and copula modeling without density estimation.
Finally, weak closure motivates
$\mathrm{cmMTP}_2$-preserving empirical,
checkerboard, and Bernstein copula approximations,
while reverse-lattice inequalities suggest
negative-dependence analogs and corresponding
Markov-composition problems.

\subsection*{Competing interests.}
No potential conflict of interest was reported by the authors.	 

\subsection*{Data availability.}
No data was used for the research described in the article.

\bibliographystyle{apalike}
\bibliography{bib}
\end{document}